\documentclass[11pt,a4paper]{article}
\usepackage{amssymb,  amsthm, color, amsfonts, amsmath}
\usepackage{graphics}
\usepackage{latexsym}
\usepackage{hyperref}
\title{\vspace{-2cm}
Analysis of Oscillations and Defect Measures for the Quasineutral Limit in Plasma Physics
}
\author{\textbf{Donatella Donatelli and Pierangelo Marcati}\\
        {\small Dipartimento di Matematica Pura ed Applicata}\\
       {\small Universit\`a degli Studi dell'Aquila}\\
       {\small 67100 L'Aquila, Italy}\\
        $\scriptstyle\mathtt{\{donatell, marcati@univaq.it\}}$\\}

\newcommand{\e}{\varepsilon}		       
\newcommand{\R}{\mathbb{R}}

\newcommand{\FF}{\mathcal{F}}
\newcommand{\ue}{u^{\varepsilon}}
\newcommand{\ut}{\tilde{u}}
\newcommand{\st}{\tilde{\sigma}}
\newcommand{\Vt}{\tilde{V}}
\newcommand{\re}{\rho^{\varepsilon}}
\newcommand{\rt}{\tilde{\rho}}
\newcommand{\se}{\sigma^{\varepsilon}}
\newcommand{\la}{\lambda}
\newcommand{\rl}{\rho^{\lambda}}
\newcommand{\sls}{\sigma^{\lambda}}
\newcommand{\vl}{V^{\lambda}}
\newcommand{\ul}{u^{\lambda}}
\newcommand{\el}{E^{\lambda}}
\newcommand{\wl}{w^{\lambda}}
\newcommand{\dive}{\mathop{\mathrm {div}}}

\newtheorem{theorem}{Theorem}[section]
\newtheorem{Main}{Main Theorem}[]
\newtheorem{corollary}[theorem]{Corollary}
\newtheorem{lemma}[theorem]{Lemma}
\newtheorem{proposition}[theorem]{Proposition}

\theoremstyle{definition}

\newtheorem{remark}[theorem]{Remark}

\begin{document}
\maketitle
\begin{abstract}
We perform a rigorous   analysis of  the  quasineutral limit for a hydrodynamical model of a viscous plasma represented by the Navier Stokes Poisson system in $3-D$. We show that as $\la\to 0$ the velocity field $\ul$ strongly converges towards an incompressible velocity vector field $u$ and the density fluctuation $\rl-1$ weakly converges to zero. In general the limit velocity field cannot be expected to satisfy the incompressible Navier Stokes equation, indeed the presence of high frequency oscillations strongly affects the quadratic nonlinearities and we have to take care of self interacting wave packets. We shall provide a detailed mathematical description  of the convergence process by using microlocal defect measures and by developing an explicit correctors analysis.  Moreover we will be able to identify an explicit pseudo parabolic pde satisfied by the leading correctors terms. Our results include all the previous results in literature, in particular we show that the formal limit holds rigorously in the case of well prepared data.

\medbreak 
\medbreak
\end{abstract}
\tableofcontents
\newpage
\section{Introduction and plan of the paper}
\subsection{Introduction}
In this paper we perform a rigorous   analysis of  the so called quasineutral limit for a hydrodynamical model of a viscous plasma represented by the Navier Stokes Poisson system in $3-D$, namely
\begin{equation}
\partial_{t}{\rho^{\lambda}}+\dive(\rho^{\lambda} u^{\lambda})=0, \label{1} 
\end{equation}
\begin{equation}
\partial_{t}(\rho^{\lambda} u^{\lambda})+\dive(\rho^{\lambda} u^{\lambda}\otimes u^{\lambda})+\nabla (\rho^{\lambda})^{\gamma}=\mu\Delta u^{\lambda}+(\nu+\mu)\nabla\dive u^{\lambda}+\rho^{\lambda}\nabla V^{\lambda}, \label{2}
\end{equation}
\begin{equation}
 \lambda^2 \Delta V^{\lambda} = \rho^{\lambda} - 1. \label{3}
\end{equation}
Let us denote by $x\in\R^{3}$, $t\geq 0$, the space and time variable, $\rho(x,t)$ the {\em negative charge density}, $m(x,t)=\rho(x,t)u(x,t)$ the {\em current density}, $u(x,t)$ the {\em velocity field}, $V(x,t)$ the {\em electrostatic potential}, $\mu, \nu$ the {\em shear viscosity} and {\em bulk viscosity} respectively. The parameter $\lambda$ is the so called {\em Debye length} (up to a constant factor).\\
We show that as $\la\to 0$ the velocity field $\ul$ strongly converges towards an incompressible velocity vector field $u$ and the density fluctuation $\rl-1$ weakly converges to zero.In general the limit velocity field cannot be expected to satisfy the incompressible Navier Stokes equation, indeed the presence of high frequency oscillations strongly affects the quadratic nonlinearities and we have to take care of self interacting wave packets.  In the paper we shall provide a detailed analysis of the convergence process by using microlocal defect measures and by developing an explicit correctors analysis. Moreover we will be able to identify an explicit pseudo parabolic equation satisfied by by the leading correctors terms.

The previous system can be seen as the coupling of the compressible Navier Stokes equations \eqref{1}, \eqref{2} with a Poisson equation \eqref{3}, where in dimensionless units the coupling constant can be expressed in terms of a parameter $\la$ which represents the scaled Debye length, which  is a characteristic physical parameter related to the phenomenon of the so called ``Debye shielding'', \cite{GR95}, studied by Peter Debye in 1912.  Any charged particle inside a plasma attracts other particles with opposite charge and repels those with the same charge, thereby creating a net cloud of opposite charges around itself, this cloud shields the particle's own charge from external view and then causes the particle's Coulomb field to fall off exponentially at large radii, rather than falling off as $1/r^{2}$.  So the physical meaning of the Debye length $\lambda$ is  the distance over which the usual Coulomb field is killed off exponentially by the polarization of the plasma. In terms of physical variables the Debye length can be expressed as
\begin{equation}
\la=\la_{D}/L  \qquad \la_{D}=\sqrt{\frac{\e_{0}k_{B}T}{e^{2}n_{0}}},
\end{equation}
where $L$ is the macroscopic length scale, $\e_{0}$ is the vacuum permittivity, $k_{B}$ the Boltzmann constant, $T$ the average plasma temperature, $e$ the absolute electron charge and $n_{0}$ the average plasma density.  In many cases the Debye length is very small compared to the macroscopic length $\la_{D} << L$ and so it makes sense to consider the quasineutral limit $\la \to 0$ of the system \eqref{1}-\eqref{3}. In this situation the particle density is constrained to be close to the background density (equal to one in our case) of the oppositely charged particle. The limit $\la \to 0$ is called the quasineutral limit since the charge density almost vanishes identically. The velocity of the fluid then evolves according to the incompressible Navier Stokes flow.
 This type of limit has been studied by many authors. In the case of Euler Poisson system by Cordier and Grenier \cite{CG00}, Grenier \cite{G95}, Cordier, Degond, Markowich and Schmeiser \cite{CDMS96}, Loeper \cite{L05}, Peng, Wang and  Yong \cite{PWY06}, in the case of the Navier Stokes Poisson system by Wang \cite{W04} and Jiang and Wang \cite{JW06} and in the contest of a combined quasineutral and relaxation time limit by Gasser and Marcati in \cite{GM01a, GM01b, GM03}.
This paper is still a mathematical theoretical approach to this complicate physical problem which however removes many regularity and smallness assumptions of various papers in the literature see for instance Wang \cite{W04} and Jiang and Wang \cite{JW06}. In fact Wang \cite{W04}	studied the	quasineutral limit for the smooth solution with well-prepared initial data. Wang and Jiang  \cite{JW06} studied the combined quasineutral and inviscid limit of the compressible Navier- Stokes-Poisson system for weak solution and obtained the convergence of Navier- Stokes-Poisson system to the incompressible Euler equations with general initial data. Moreover in  \cite{JW06} the vanishing of viscosity coefficient was required in order to take the quasineutral limit and no convergence rate was derived therein. The authors in  \cite{DM08}  investigated the quasineutral limit of the isentropic Navier-Stokes-Poisson system in the whole space  and obtained the convergence of weak solution of the Navier-Stokes-Poisson system to the weak solution of the incompressible Navier- Stokes equations by means of dispersive estimates of Strichartz's type under the assumption that the Mach number is related to the Debye length. Ju, Li and Wang \cite{JLW08} studied the quasineutral limit of the isentropic Navier-Stokes-Poisson system both in the whole space and in the torus without the restriction on viscous coefficient with well prepared initial data. However there is no analysis for the quasineutral limit for the Navier Stokes Poisson system in the context of weak solutions and in the framework of general ill prepared initial data. The common feature of this kind of limits in the ill prepared data framework is the high plasma oscillations, namely the presence of high frequency  time oscillations along the acoustic waves. 
In these phenomena there are different behaviors of the various vector fields acting in our system. Particularly relevant us to understand the relationship between high frequency interacting waves, dispersive behavior and the different role of time and space oscillations.  In our analysis the velocity fields both disperse and oscillates however the dispersion behavior dominates on the high frequency time oscillations and Strichartz estimates are sufficient to pass into the limit of the convective term. The presence of quadratic terms on the electric field (e.g. $\rl\nabla\vl$) cannot be analyzed in the same way since the dispersive behavior no longer dominates on time hight frequency wave packets. In the general case these quadratic terms will not vanish in the limit as $\la\to 0$, unless we have well prepared initial data.

\subsection{Plan of the paper}
The structure of this paper, as well as the main ingredients of our approach to this limiting process can be summarized as follows.
\begin{itemize}
\item In Section 2 we collect many needed mathematical tools, including notations, Strichartz estimates and microlocal defect measures. Then in Section 3 we set up our problem.
\item  The following section 5 is devoted to obtain a priori estimates independent of $\la$, namely standard energy bounds and dispersive estimates on the density fluctuation. The main idea here is based on the observation that the density fluctuation $\rl -1$ satisfies a Klein-Gordon equation, so acoustic waves analysis for the Navier Stokes Poisson system  \eqref{1}-\eqref{3} follows by reading the system as a dispersive  equation and we will get uniform estimates in $\lambda$ by the use of he $L^{p}$-type estimates due to Strichartz \cite{GV95, KT98, S77}. The particular type of  Strichartz estimates for the Klein Gordon equation that we  use here can be recovered from the seminal paper by Strichartz \cite{S77} (where he studied the homogenous equation) and by using Duhamel's principle. 
\item In the previous sections  we get sufficient bounds in order to study the limiting behaviour of the velocity vector field. Therefore in Section 5 we analyze separately the limiting behaviour of the divergence free part and the gradient part of $u^{\lambda}$. Accordingly we obtain the strong convergence of the velocity field 
\item The next stumbling block is to get enough compactness for the electric field in order to pass into the limit in the quadratic term $\la\nabla\vl\otimes\la\nabla\vl$. Since $\la\nabla\vl$ is bounded in $L^{\infty}_{t}L^{2}_{x}$ we can define microlocal defect measure $\nu^{E}$ introduced by P. G\`erard in \cite{Ger91} and by L.Tartar (H-measure) in \cite{Tar90} with correctors $E^{+}$ and $E^{-}$ to handle time oscillations at frequency $1/\la$. An analogous use of the P. G\`erard  and L. Tartar ideas can be found in Y. Brenier and E.Grenier \cite{BG94} and E. Grenier \cite{G95a}, regarding the Vlasov Poisson system. This will be done in Section 6.
\item In Section 7  we will be able to prove our Main Theorem \ref{tm1}.
\item As a final step, in Section 8, we show that in the case of smooth solutions for the system \eqref{1}-\eqref{3} the class of correctors $E^{+}$ and $E^{-}$ is not empty and they satisfy a ``pseudo parabolic'' type equation, see the Main Theorem \ref{tm2}.
\end{itemize}

\section{Preliminaries}
For convenience of the reader we establish some notations and recall some basic facts that will be useful in the sequel.
\subsection{Notations}
If $F,G$ are functions we denote by $F\lesssim G$ the fact that there exists $c\in \R$ such that $F\leq G$. Then, we will denote by  
\begin{itemize}
\item[a)] $\mathcal{D}(\R ^d \times \R_+)$ the space of test function $C^{\infty}_{0}(\R^d \times \R_+)$, by $\mathcal{D}'(\R^d \times\R_+)$ the space of Schwartz distributions and $\langle \cdot, \cdot \rangle$ the duality bracket between $\mathcal{D}'$ and $\mathcal{D}$
\item[b)] $W^{k,p}(\R^{d})=(I-\Delta)^{-\frac{k}{2}}L^{p}(\R^{d})$ and $H^{k}(\R^{d})=W^{k,2}(\R^{d})$ the nonhomogeneous Sobolev spaces, for any $1\leq p\leq \infty$ and $k\in \R$. $\dot W^{k,p}(\R^{d})=(-\Delta)^{-\frac{k}{2}}L^{p}(\R^{d})$ and $\dot H^{k}(\R^{d})=W^{k,2}(\R^{d})$  denote the homogeneous Sobolev spaces. The notations $L^{p}_{t}L^{q}_{x}$ and $L^{p}_{t}W^{k,q}_{x}$ will abbreviate respectively  the spaces $L^{p}([0,T];L^{q}(\R^{d}))$, and $L^{p}([0,T];W^{k,q}(\R^{d}))$.
\item[c)] $L^{p}_{2}(\R^{d})$ the Orlicz space defined as follows
\begin{equation}
\label{2.1}
L^{p}_{2}(\R^{d})=\{f\in L^{1}_{loc}(\R^{d})\mid |f|\chi_{|f|\leq \frac{1}{2}}\in L^{2}(\R^{d}),\ |f|\chi_{|f|> \frac{1}{2}}\in L^{p}(\R^{d})\},
\end{equation}
see \cite{Ada75}, \cite{LPL96} for more details.
\item[d)]  $ \mathcal{L}(\R^{3})$ the space of
bounded operators, $ \mathcal{K}(\R^{3})$ the space of compact operators,
\item[ e)] if $X$, $Y$ are Banach spaces, $ \mathcal{L}(X,Y)$ is the 
space of bounded operators
\item[f)] $Q$ and $P$ respectively  the Leray's projectors $Q$ on the space of gradients vector fields and $P$ on the space of divergence - free vector fields. Namely
\begin{equation}
Q=\nabla \Delta^{-1}\dive, \qquad P=I-Q.
\label{pr}
\end{equation} 
It is well known that   $Q$ and $P$  can be expressed in terms of Riesz multipliers, therefore they are  bounded linear operators on every $W^{k,p}$ $(1< p<\infty)$ space (see \cite{Ste93}).   \\ \\
\end{itemize}
Next we recall the basic notations concerning pseudo-differential operators and symbols to be used
later on. We refer to \cite{Tay91} for details. Assuming $\rho, \delta
\in [0,1],\ m \in \R$, we denote $S^{m}_{\rho, \delta}$ the set of $C^\infty$ symbols satisfying $$\left| D^\beta _x D^\alpha _\xi p(x,\xi)\right|\leq C_{\alpha, \beta}\langle \xi \rangle
^{m-\rho|\alpha|+\delta|\beta|}$$ for all $\alpha, \beta$, where
$\langle \xi \rangle =(1+|\xi|^2)^{1/2}$. In such case we say that the
associated operator denoted by $OP(p(x,\xi))$ is given by 
$$P(x,D)f(x)=\int p(x, \xi) \mathcal{F}f(\xi) e^{i x \xi}d\xi:=OP(p(x,\xi))$$ 
(where$\mathcal{F}f
(\xi)=(2\pi)^{-n}\int f(x) e^{-ix \xi}dx $ denotes the Fourier transform of
the function $f$) belongs to $OPS^{m}_{\rho, \delta}$. If there
are smooth symbols $p_{m-j}(x,\xi)$, homogeneous in $\xi$ of degree $m-j$ for
$|\xi|\geq 1$, i.e.  $p_{m-j}(x,r\xi)=r^{m-j}p_{m-j}(x,\xi)$ for $r>0,\
|\xi|\geq 1$, and if $$p(x,\xi)\sim \sum_{j\geq 0} p_{m-j}(x,\xi)$$ in
the sense that $$p(x,\xi)-\sum_{j\geq 0}^{N} p_{m-j}(x,\xi)\in
S^{m-N}_{1,0}
$$ for all $N$, then we say $p(x,\xi)\in S^m$ and $P(x,D)$ is polyhomogenous of order $m$. If $\Omega$ is an open set in $\R^{3}$, we denote by $\psi_{comp}^{m}(\Omega,\mathcal{L}(H))$, respectively, $\psi_{comp}(\Omega,\mathcal{K}(H))$ the space of polyhomogenous pseudo-differential operators of order $m$ on $\Omega$, with values in $\mathcal{L}(H)$, respectively $\mathcal{K}(H)$ whose kernel is compactly supported in $\Omega\times\Omega$, moreover we recall that if $P\in \psi_{comp}^{m}(\Omega,\mathcal{L}(H))$, then its symbol $p(x,\xi)$ is \ a \  linear application from $\psi_{comp}^{m}(\Omega,\mathcal{L}(H))$ to $C^{\infty}_{0}(S^{*}\Omega, \mathcal{L}(H))$, where $S^{*}\Omega=S^{d-1}\times\Omega$.

Following P. G\`erard we say that $\mu$ is the {\em microlocal defect measure} (or following L. Tartar the {\em H-measure}) for a bounded sequence $w_{k}$ in $L^{2}$ if for any $A\in \psi_{comp}^{0}(\omega, \mathcal{K}(H))$ one has (up to subsequences)
$$\lim_{k\to \infty}(A(w_{k}-w), (w_{k}-w))=\int_{S^{\ast}\Omega}tr(a(x,\xi)\mu(dxd\xi)).$$
where $A=OP(a(x,\xi))$.

\subsection{Technical tools}
\subsubsection{Strichartz estimates for Klein Gordon equations}

Let us recall that if  $w$ is a solution of the following Klein Gordon  equation in the space $[0,T]\times \R^{d}$
\begin{equation*}
\left(-\frac{\partial ^{2}}{\partial t}+\Delta-m^{2}\right)w(t,x)=F(t,x)\\
\end{equation*}
with Cauchy data 
\begin{equation*}
w(0,\cdot)=f,\quad \partial_{t}w(0,\cdot)=g,
\end{equation*}
where $m>0$ is the mass and $0<T<\infty$, 
then $w$ satisfies the following Strichartz estimates, (see \cite{S77})
\begin{equation*}
\|w\|_{L^{q}_{t,x}}+\|\partial_{t}w\|_{L^{q}_{t}W^{-1,q}_{x}}\lesssim \|f\|_{\dot H^{1/2}_{x}}+\|g\|_{\dot H^{-1/2}_{x}}+\|F\|_{L^{p}_{t,x}},
\label{s2'}
\end{equation*}
where $(q,p)$,  are \emph{admissible pairs}, namely they satisfy 
\begin{equation*}
\frac{2(n+1)}{n+3}\leq p\leq\frac{2(n+2)}{n+4} \qquad 
\frac{2(n+2)}{n}\leq p\leq\frac{2(n+1)}{n-1}.
\end{equation*}
In particular in the case of $d=3$, $(q,p)$ are admissible if they satisfy
\begin{equation*}
\frac{4}{3}\leq p \leq \frac{10}{7} \qquad \frac{10}{3}\leq q \leq \frac{4}{3}. 
\end{equation*}
Moreover by choosing $p=4/3$ and $q=4$ and by a standard application of Duhamel's principle we have the following estimate
\begin{equation}
\|w\|_{L^{4}_{t,x}}+\|\partial_{t}w\|_{L^{4}_{t}W^{-1,4}_{x}}\lesssim \|f\|_{\dot H^{1/2}_{x}}+\|g\|_{\dot H^{1/2}_{x}}+\|F\|_{L^{1}_{t}L^{2}_{x}}.
\label{s1}
\end{equation}
It is straightforward to observe that for any $s\geq 0$ also this  estimate hold
\begin{equation}
\|w\|_{L^{4}_{t}W^{-s,4}_{x}}+\|\partial_{t}w\|_{L^{4}_{t}W^{-1-s,4}_{x}}\lesssim \|f\|_{ H^{1/2-s}_{x}}+\|g\|_{ H^{-1/2-s}_{x}}+\|F\|_{L^{1}_{t}H^{-s}_{x}}.
\label{s2}
\end{equation}
(it is sufficient to  apply the operator $(I-\Delta)^{-s/2}$ to \eqref{s1}).

\subsubsection{Properties for pseudo-differential operators}
We recall here two fundamental tools necessary to work with pseudodifferential operators (fore more details see \cite{Tay91}, \cite{Tay81},  \cite{Ger91}, )
\begin{proposition}
\label{p2.2}
If $A\in OPS^{0}$, then 
$$A:L^{2}_{loc}(\Omega, H)\rightarrow L^{2}(\Omega, H)$$
is bounded.
\end{proposition}
\begin{proposition}[\bf{Generalized Rellich Theorem}]
\label{p2.3}
If $A\in \psi_{comp}^{m}(\Omega, \mathcal{K}(H))$ for some $m<0$, then
$$A:L^{2}_{loc}(\Omega, H)\rightarrow L^{2}(\Omega, H)$$
is compact, i.e. if $w_{k}\rightharpoonup w$ weakly in $L^{2}$, then $\|Aw_{k}-w\|\to 0$ strongly. 
\end{proposition}

\subsubsection{Convolution estimate and compactness theorems}
Here we state the following elementary lemma that will be used later on.
\begin{lemma}
Let us consider  a smoothing kernel $j\in C^{\infty}_{0}(\R^{d})$, such that $j\geq 0$, $\int_{\R^{d}}j dx=1$, and let us define
\begin{equation*}
j_{\alpha}(x)=\alpha^{-d}j\left(\frac{x}{\alpha}\right).
\end{equation*}
Then  for any $f\in \dot H^{1}(\R^{d})$, one has
\begin{equation}
\label{y1}
\|f-f\ast j_{\alpha}\|_{L^{p}(\R^{d})}\leq C_{p}\alpha^{1-d\left(\frac{1}{2}-\frac{1}{p}\right)}\|\nabla f\|_{L^{2}(\R^{d})},
\end{equation}
where
\begin{equation*}
p\in [2, \infty)
\quad \text{if $d=2$}, \quad p\in [2, 6] \quad \text{if $d=3$}.
\end{equation*}
Moreover the following Young type inequality hold
\begin{equation}
\label{y2}
\|f\ast j_{\alpha}\|_{L^{p}(\R^{d})}\leq C\alpha^{s-d\left(\frac{1}{q}-\frac{1}{p}\right)}\|f\|_{W^{-s,q}(\R^{d})},
\end{equation}
for any $p,q\in [1, \infty]$, $q\leq p$,  $s\geq 0$, $\alpha\in(0,1)$.
\label{ly}
\end{lemma}
We recall also the following compactness tool (see \cite{Si}).
\begin{theorem}
Let be $\mathcal{F}\subset L^{p}([0,T];B)$,  $1\leq p<\infty$, $B$ a Banach space. $\mathcal{F}$ is relatively compact in  $L^{p}([0,T];B)$ for $1\leq p<\infty$, or in $C([0,T];B)$ for $p=\infty$ if and only if 
\begin{itemize}
\item[{\bf (i)}]
$\displaystyle{\left\{\int_{t_{1}}^{t_{2}}f(t)dt,\ f\in \mathcal{F}\right\}}$ is relatively compact in $B$, $0<t_{1}<t_{2}<T$,
\item[{\bf (ii)}]
$\displaystyle{\lim_{h\to 0}\|f(t+h) - f(t)\|_{L^{p}([0, T-h];B)}=0}$ uniformly for any $f \in \mathcal{F}$.
\end{itemize}
\label{S}
\end{theorem}

\section{Statement of the problem and Main Results}

\subsection{Basic facts on the Navier Stokes Poisson System}
In order set up our problem we recall here some  results concerning the existence theory for the Navier Stokes Poisson system \eqref{1}-\eqref{3}. For simplicity we rewrite here again the compressible Navier Stokes equation coupled with the Poisson  equation
\begin{equation}
\label{3.1}
\begin{cases}
\partial_{t}{\rho^{\lambda}}+\dive(\rho^{\lambda} u^{\lambda})=0\\
\partial_{t}(\rho^{\lambda} u^{\lambda})+\dive(\rho^{\lambda} u^{\lambda}\otimes u^{\lambda})+\nabla (\rho^{\lambda})^{\gamma}=\mu\Delta u^{\lambda}+(\nu+\mu)\nabla\dive u^{\lambda}+\rho^{\lambda}\nabla V^{\lambda}\\
\lambda^{2}\Delta V^{\lambda}=\rho^{\lambda}-1.
\end{cases}
\end{equation}
To simplify our notation from now on we will set 
$$\pi^{\lambda}=\frac{(\rl)^{\gamma}-1-\gamma(\rl-1)}{(\gamma-1)} \qquad \mu=\nu=1.$$
The system \eqref{3.1} is endowed with the following initial conditions,
\begin{align}
\tag {\bf{ID}}
&\rho^{\lambda}_{t=0}=\rho^{\lambda}_{0}\geq 0, \  V^{\lambda}|_{t=0}=V_{0}^{\lambda}\notag,\\
&\rl\ul |_{t=0}=m^{\lambda}_{0}, \quad m^{\lambda}_{0}=0 \ \text{on}\ \{x\in \R^{3}\mid \rl_{0}(x)=0\},\notag\\
&\int_{\R^{3}}\left(\pi^{\la}|_{t=0}+\frac{|m^{\la}_{0}|^{2}}{2\rl_{0}}+\la^{2}|V^{\la}_{0}|\right)dx\leq C_{0}.\notag
\end{align}
The existence of global weak solutions for fixed $\la>0$ for the system \eqref{3.1}, has been proved in the case of a bounded domain in \cite{Don03} and in the case of the whole domain in \cite{DFPS01} and \cite{DFPS04}. We summarize this existence result in the following theorem.
\begin{theorem}
\label{t1}
Assume (ID),  and let $\gamma>3/2$, then there exists a global weak solution $(\rl, \ul, \vl)$ to \eqref{3.1} such that $\rl-1\in L^{\infty}((0,T);L^{\gamma}_{2}(\R^{3}))$, $\sqrt{\re}\ue\in L^{\infty}((0,T);L^{2}(\R^{3}))$, $\ul\in L^{2}((0,T);W^{1,2}(\R^{3}))$. 
Furthermore
\begin{itemize}
\item The energy inequality holds for almost every $t\geq0$,
\begin{align}
\label{2.5}
\int_{\R^{3}}&\left(\rl\frac{|\ul|^{2}}{2}+\pi^{\la}+\la^{2}|\nabla V^{\la}|^{2}\right)dx\notag\\
&+\int_{0}^{t}\int_{\R^{3}}\left(\mu|\nabla\ul|^{2}+(\nu+\mu)|\dive\ul|^{2}\right)dxds\leq C_{0}.
\end{align}
\item The continuity equation is satisfied in the sense of renormalized solutions, i.e.:
\begin{equation*}
\partial_{t}b(\rl)+\dive(b(\rl)u)+(b'(\rl)\rl-b(\rl))\dive\ul=0,
\end{equation*}
for any $b\in C^{1}(\R^{3})$ such that
$$b'(z)=\text{constant}, \quad \text {for any $z$ large enough, say $z\geq M$}.$$
\item The system \eqref{3.1} holds in $\mathcal{D'}((0,T)\times\R^{3})$.
\end{itemize}
 \end{theorem}
Beside the results on the existence of weak solutions for the Cauchy problem for the Navier Stokes Poisson system \eqref{3.1} there is a theory concerning the global existence of classical solutions of \eqref{3.1} see for example \cite{MZ10} for the $H^{s}$ framework or \cite{HL09} for global solutions in Besov spaces. We describe this global existence result in the following theorem.
\begin{theorem}
Assume that $(\rl_{0}-1, m_{0})\in H^{s}(\R^{3})\cap L^{1}(\R^{3})$, $s\geq 4$, with $\delta=\|(\rl_{0}-1, m_{0})\|_{H^{s}(\R^{3})\cap L^{1}(\R^{3})}$ small. Then, there is a unique global classical solution $(\rl, m^{\la},\vl)$ to the system \eqref{3.1} satisfying
$$\rl-1\in C^{0}(\R_{+}, H^{s}(\R^{3}))\cap C^{1}(\R_{+}, H^{s-1}(\R^{3})),$$
$$m\in C^{0}(\R_{+}, H^{s}(\R^{3}))\cap C^{0}(\R_{+}, H^{s-2}(\R^{3})),$$
$$\la\vl \in C^{0}(\R_{+}, L^{6}(\R^{3}))\quad \la\nabla \vl\in C^{0}(\R_{+}, H^{s+1}(\R^{3})).$$
\label{t2}
\end{theorem}

\subsection{Main results}
Having collected all the preliminary material we are now ready to state our main results.
The first result concerns the convergence of solutions of the system \eqref{3.1} in the quasineutral regime.
\begin{Main}
Let $(\rl, \ul, V^{\la})$ be a sequence of weak solutions in $\R^{3}$ of the system (\ref{3.1}), assume that the initial data satisfy (ID). Then
\begin{itemize}
\item [\bf{(i)}] $\rl\longrightarrow 1$ \quad weakly in $L^{\infty}([0,T];L^{k}_{2}(\R^{3}))$.
  \item [\bf{(ii)}] There exists $u\in L^{\infty}([0,T];L^{2}(\R^{3}))\cap L^{2}([0,T];\dot H^{1}(\R^{3}))$ such that 
  \begin{equation*}
\ul\rightharpoonup u \quad \text{weakly in $L^{2}([0,T];\dot H^{1}(\R^{3}))$}.
\end{equation*}
  \item [\bf{(iii)}] The gradient component $Q\ul$ of the vector field $\ul$ satisfies
  \begin{equation*}
Q\ul\longrightarrow 0\quad \text{ strongly in $L^{2}([0,T];L^{p}(\R^{3}))$, for any $p\in [4,6)$}.
\end{equation*}
 \item [\bf{(iv)}] The divergence free component $P\ul$ of the vector field $\ul$ satisfies
   \begin{equation*}
P\ul\longrightarrow Pu=u\quad \text{strongly  in $L^{2}([0,T];L^{2}_{loc}(\R^{3}))$}.
\end{equation*}
\item [\bf{(v)}] There exist correctors $E^{+}$, $E^{-}$  in $L^{\infty}((0,T),L^{2}(\R^{3}))$ and a  positive microlocal defect measure $\nu^{E}$ on $\R^{3}\times S^{2}$ depending measurably on $t$, associated to the electric field $E^{\la}=\nabla\vl$, such that for all pseudodifferential operators $A\in \psi_{comp}^{0}(\R^{3},\mathcal{K}(\R^{3}))$, and of symbol $a(x,\xi)$ and for all $\phi\in \mathcal{D}(0,t)$ one has
\begin{align}
\lim_{\la\to 0}\int dt\phi(t) \la^{2}(A\el,\el)&=\int dt\phi(t) (AE^{+},E^{+})+\int dt\phi(t) (AE^{-},E^{-})\notag\\
&+\int dt\phi(t) \!\!\int_{\R^{3}\times S^{2}}tr\left(a(x,\xi)\frac{\xi\otimes \xi}{|\xi|^{2}}\right)d\nu^{E}.
\label{6.1a}
\end{align}
 \item [\bf{(vi)}] $u=Pu$ satisfies the following equation
\begin{align}
P\Big(\partial_{t} u&-\Delta u+(u\cdot\nabla)u- \notag\\
& \dive(E^{+}\otimes E^{+}+E^{-}\otimes E^{-})-\dive\langle \nu^{E}, \frac{\xi\otimes \xi}{|\xi|^{2}}\rangle\Big)=0,
\label{nsd}
\end{align}
in $\mathcal{D}'([0,T]\times \R^{3})$.
\end{itemize}
\label{tm1}
\end{Main}
\begin{remark}
\label{r1}
In the previous theorem we constructed a defect measure $\nu^{E}$ and the correctors $E^{\pm}$. They correspond to the physical phenomenon of the high frequency plasma oscillation. Notice that the correctors $E^{\pm}$ remain important as $\la\to 0$ and are not vanishing, in fact we  don't have initial layer but on the contrary the effect of ill prepared initial data appears through $E^{\pm}$ and remains important for all times.
\end{remark}
As we will see in the rest of paper the construction of the defect measure $\nu^{E}$ will be done by using the theory developed by P. G\`erard in \cite{Ger91} and L. Tartar in \cite{Tar90}. The explicit construction of the correctors is not trivial and requires a smooth setting for the solutions. We will show this part in the next theorem.
\begin{Main}
\label{tm2}
Let be $(\rl, \ul, V^{\la})$ be a sequence of the Navier Stokes Poisson system, satisfying for $s\geq 4$ 
\begin{equation}
\|\rl-1\|_{L^{\infty}(0,T;H^{s}(\R^{3}))}\leq C \qquad \|\la \el\|_{L^{\infty}(0,T;H^{s}(\R^{3}))}\leq C
\end{equation}
then, for all $s'<s-2$
\begin{equation}
\ul-\frac{1}{i}e^{-it/\la}E^{+}-\frac{1}{i}e^{it/\la}E^{-}\longrightarrow v
\quad \text {strongly in $C^{0}(0,T, H^{s'-1}_{loc}(\R^{3}))$.}
\end{equation}
\begin{equation}
\la(\el-e^{-it/\la}E^{+}-e^{it/\la}E^{-})\longrightarrow 0
\quad \text {strongly in $C^{0}(0,T, H^{s'-1}_{loc}(\R^{3}))$.}
\end{equation}
and $E^{\pm}$ satisfy
\begin{equation}
\label{ceq}
\partial_{t} E^{\pm}-\Delta E^{\pm}+Q \dive(v\otimes E^{\pm})=0, \qquad PE^{\pm}=0.
\end{equation} 
\end{Main}
In the Section 8 we will show in the Proposition \ref{p8} the existence of solutions for the equation \eqref{ceq}. The rest of the paper is devoted to prove the Main Theorems \ref{tm1} and \ref{tm2}.

\section{Uniform estimates}
In this section we wish to establish all the a priori estimates, independent on $\la$, for the solutions of the system \eqref{3.1} which are necessary to prove the Main Theorem \ref{tm1}. 

\subsection{Consequences of the energy estimate}
We start by collecting  all the a priori bounds that are a consequence of the energy inequality \eqref{2.5}.
Before going on let us define the density fluctuation $\sls$ as
\begin{equation}
\label{3.1.1}
\sls=\rl-1.
\end{equation}
\begin{proposition}
\label{p4.1}
Let us consider the solution $(\rl,\ul, \vl)$ of the Cauchy problem for the system \eqref{3.1}. Assume that the hypotheses (ID) hold, then it follows
\begin{align}
& \sls&\qquad&\text{is bounded in $L^{\infty}([0,T];L^{k}_{2}(\R^{3}))$, where $k=\min(\gamma,2)$},\label{4.1.1}\\
&\nabla\ul &\qquad&\text{is bounded in $L^{2}([0,T]\times\R^{3})$},\label{4.1.2bis}\\
&\ul &\qquad&\text{is bounded in $L^{2}([0,T]\times\R^{3})\cap L^{2}([0,T];L^{6}(\R^{3}))$},\label{4.1.2}\\
&\sls\ul&\qquad&\text{is bounded in $L^{2}([0,T];H^{-1}(\R^{3}))$},\label{4.1.3}\\
&\la\nabla\vl&\qquad&\text{is bounded in $L^{\infty}([0,T];L^{2}(\R^{3}))$}.\label{4.1.3b}
\end{align}
\end{proposition}
\begin{proof}
 From \eqref{2.5} it follows that $\pi^{\la}\in L^{\infty}([0,T];L^{1}(\R^{3}))$.   By taking into account that the function $z\rightarrow z^{\gamma}-1-\gamma(z-1)$ is convex and by following the same line of arguments as in \cite{L-P.L.M98} we get when $\gamma<2$ that
\begin{equation}
\sup_{t\geq 0}\int_{\R^{3}}\left\{|\rl-1|^{2}\chi_{|\rl-1|\leq1/2}+|\re-1|^{\gamma}\chi_{|\rl-1|\geq1/2}\right\}(t,x)dx\leq C
\label{4.1.4}
\end{equation}
and when $\gamma\geq2,$
\begin{equation}
\sup_{t\geq 0}\int_{\R^{3}}|\rl-1|^{2}(t,x)dx\leq C,
\label{4.1.5}
\end{equation}
so we can conclude that $\sls$ is uniformly bounded in $\la$ in $L^{\infty}([0,T];L^{k}_{2}(\R^{3}))$, where $k=\min(\gamma,2)$. \eqref{4.1.2bis} and \eqref{4.1.3b} are a consequence of \eqref{2.5}. The fact that $\ul\in L^{2}([0,T];L^{6}(\R^{3}))$ follows from \eqref{4.1.2bis} and by Sobolev's embeddings.
Now we prove $\ul\in L^{2}([0,T]\times\R^{3})$. 
\begin{align}
\int_{\R^{3}}|\ul|^{2}dx&=\int_{\R^{3}}\left\{|\ul|^{2}\chi_{|\rl-1|\leq1/2}+|\ul|^{2}\chi_{|\re-1|\geq1/2}\right\}dx\notag\\
&\leq 2\int_{\R^{3}}\rl |\ul|^{2}dx+2\|\re-1\|_{L^{k}_{x}}\|\ul\|_{L^{2k/k-1}_{x}}^{2}\notag\\
&\leq C_{0}+C_{0}\|\ue\|^{2-\frac{3}{k}}_{L^{2}_{x}}\|\nabla\ul\|^{\frac{3}{k}}_{L^{2}_{x}}
\label{4.1.6}
\end{align}
We complete then easily the prove by using \eqref{4.1.2bis}.
Recalling that $\gamma>3/2$ and by interpolating we get that $\ul\in L^{2}([0,T];L^{4}(\R^{3})\cap L^{2\gamma/(\gamma-1)}(\R^{3}))$. By using \eqref{4.1.1} we obtain that   $\rl\ul$ is uniformly bounded in $L^{2}([0,T];L^{4/3}(\R^{3})+L^{2k/(k+1)}(\R^{3}))$. Therefore by Sobolev's embeddings we get \eqref{4.1.3}.
\end{proof}
We want to complete this paragraph with a remark concerning the regularity of the initial data. 
\begin{remark}
\label{r2}
With the same procedure as for $\sls$, taking into account (ID) we get that $\sls_{0}$ is bounded in $L^{k}_{2}(\R^{3})$ hence in $H^{-1}(\R^{3})$, since $\gamma>3/2$.
If we rewrite $m^{\la}_{0}$ in the following way 
$$m^{\la}_{0}=\frac{m^{\la}_{0}}{\sqrt{\rl_{0}}}\sqrt{\rl_{0}}\chi_{|\rl_{0}-1|\leq 1/2}+\frac{m^{\la}_{0}}{\sqrt{\rl_{0}}}\frac{\sqrt{\rl_{0}}}{\sqrt{|\rl_{0}-1|}}\sqrt{|\rl_{0}-1|}\chi_{|\rl_{0}-1|> 1/2}$$
we get that $m^{\la}_{0}$ is bounded in $L^{2}(\R^{3})+L^{2k/(k+1)}(\R^{3})$ and hence in $H^{-1}(\R^{3})$. Finally we can conclude that
\begin{equation}
\sls_{0},\ m^{\la}_{0}\qquad \text{are bounded in $H^{-1}(\R^{3})$ uniformly in $\la$}.
\label{4.1.7}
\end{equation}
\end{remark}

\subsection{Density fluctuation acoustic equation}
From the estimates of the Proposition \ref{p4.1} we get only the weak convergence of the velocity field and unfortunately this will be not sufficient to pass into the limit in the nonlinear terms (such as the convective term $\dive(\rl\ul\otimes\ul)$) of the system \eqref{3.1}. In particular this weak convergence is induced by the rapid time oscillation of the acoustic waves or by the so called plasma oscillations. In order to overcome this problem we will estimate the density fluctuation $\sls$ uniformly with respect to $\la$. So we derive the so called acoustic equation which governs the time evolution of $\sls$. 
First of all we rewrite the system \eqref{2.5} in the following way
\begin{align}
\partial_{t}\sls+\dive(\rl \ul)&=0\label{4.2.1}\\
\partial_{t}(\rl\ul)+\nabla\sls&=\mu\Delta \ul+(\nu +\mu)\nabla\dive \ul-\dive(\rl \ul\otimes \ul)\notag\\
&-(\gamma-1)\nabla\pi^{\la}+\sls\nabla V^{\la}+\nabla V^{\la},
\label{4.2.2}\\
\la^{2}\Delta V^{\la}&=\sls.
\label{4.2.2bis}
\end{align}
Then, by differentiating with respect to time the equation \eqref{4.2.1},taking the divergence of \eqref{4.2.2} and by using \eqref{4.2.2bis} we get that $\sls$ satisfies the following equation
\begin{align}
\partial_{tt}{\sls}-\Delta \sls+\frac{\sls}{\la^{2}}=&-\dive(\mu\Delta \ul+(\nu +\mu)\nabla\dive \ul)\label{4.2.3}\\
&+\dive\left(\dive(\rl \ul\otimes \ul)+(\gamma-1)\nabla\pi^{\la}+\sls\nabla V^{\la}\right).\notag
\end{align}
It turns out that \eqref{4.2.3} is a nonhomogeneous Klein Gordon equation with mass $1/\la$. In order to get some more uniform estimates on $\sls$ we apply to \eqref{4.2.3} the Strichartz estimates \eqref{s2}.  
To renormalize the mass of  the equation \eqref{4.2.3} more easier to handle  we rescale the time and space variable, the density fluctuation, the velocity and the electric potential in the following way 
\begin{align}
\tau&=\frac{t}{\la}, \quad y=\frac{x}{\la}\label{4.2.4}\\ 
\ut(y,\tau)&=\ul(\la y,\la\tau), \quad  \rt(y,t)=\rl(\la y,\la\tau)\notag\\
\st(y,\tau)&=\sls (\la y,\la\tau),
\quad \Vt(y,\tau)=V^{\la}(\la y,\la\tau).\label{4.2.5}
\end{align}
As a consequence of this scaling the Klein Gordon equation \eqref{4.2.3} becomes of mass equal to one, namely
\begin{align}
\partial_{\tau\tau}{\st}-\Delta \st+\st&=-\frac{1}{\la}\dive(\mu\Delta \ut+(\nu +\mu)\nabla\dive \ut)\notag\\
&+\dive\left(\dive(\rt \ut\otimes \ut)+(\gamma-1)\nabla\tilde{\pi}+\st\nabla \Vt\right).\label{4.2.6}
\end{align}
Now we consider $\st=\st_{1}+\st_{2}+\st_{3}$ where $\st_{1}$, $\st_{2}$, $\st_{3}$  solve the following Klein Gordon equations
\begin{equation}
\label{4.2.7}
\begin{cases}
     \partial_{\tau\tau}\st_{1}-\Delta\st_{1} +\st_{1}=-\frac{1}{\la}\dive(\mu\Delta \ut+(\nu +\mu)\nabla\dive\ut)=F_{1} \\
     \st_{1}(x,0)=\st(x,0)=\st_{0}\quad \partial_{\tau} \st_{1}(x,0)=\partial_{\tau}\st(x,0)=\partial_{t}\st_{0},
   \end{cases}
\end{equation}    
\begin{equation}
\label{4.2.8} 
\begin{cases}   
\partial_{\tau\tau}\st_{2}-\Delta\st_{2}+\st_{2} =\dive(\dive(\rt\ut\otimes\ut)+(\gamma-1)\nabla\tilde{\pi})=F_{2}\\
\st_{2}(x,0)=\partial_{\tau}\st_{2}(x,0)=0,
 \end{cases}
 \end{equation}
\begin{equation}
\label{4.2.9}
\begin{cases}
     \partial_{\tau\tau}\st_{3}-\Delta\st_{3}+\st_{3} =-\dive(\st\nabla \Vt)=F_{3} \\
     \st_{3}(x,0)=\partial_{\tau} \st_{3}(x,0)=0,
   \end{cases}
\end{equation}  
We are able to prove the following estimate on $\sls$.
\begin{theorem}
Let us consider the solutions $(\rl, \ul, V^{\la})$ of the Cauchy problem for the system \eqref{3.1} with initial data satisfying (ID). Then for any $s_{0}\geq 3/2$,  the following estimate holds
\begin{align}
&\la^{-\frac{1}{2}}\|\sls\|_{L^{4}_{t} W^{-s_{0}-2,4}_{x}}
+\la^{-\frac{1}{2}}\|\partial_{t}\sls\|_{L^{4}_{t} W^{-s_{0}-3,4}_{x}}\notag\\
&\lesssim \la^{s_{0}-\frac{1}{2}}\|\sls_{0}\|_{H^{-3/2}_{x}}+\la^{s_{0}-\frac{1}{2}}\|m^{\la}_{0}\|_{H^{-5/2}_{x}}\notag\\
&+T\|\dive(\dive( \rl\ul\otimes\ul)-(\gamma-1)\nabla\pi^{\la})\|_{L^{\infty}_{t}H^{-s_{0}-2}_{x}}\notag\\
&+\la^{s_{0}}\|\dive \Delta\ul+\nabla\dive\ul\|_{L^{2}_{t}H^{-2}_{x}}+
T\|\dive (\sls V^{\la})\|_{L^{\infty}_{t}H^{-s_{0}-2}_{x}}.
\label{4.2.11}
\end{align}
\end{theorem}
\begin{proof}
Since $\st_{1}$, $\st_{2}$, $\st_{3}$,  are solutions of the  equations \eqref{4.2.7}, \eqref{4.2.8}, \eqref{4.2.9},  we can apply the Strichartz estimate \eqref{s2} with $(y,\tau)\in \R^{3}\times(0,T/\la)$. We start with $\st_{1}$. From \eqref{4.1.2bis} we deduce that $F_{1}\in L^{2}_{t}H^{-2}_{x}$, so by using \eqref{s2} with $s=2$ we get
\begin{align}
\|\st_{1}\|_{L^{4}_{\tau} W^{-2,4}_{y}}+\|\partial_{\tau}\st_{1}\|_{L^{4}_{\tau} W^{-3,4}_{y}}&\lesssim \|\st_{0}\|_{H^{-3/2}_{y}}+\|\partial_{\tau}\st_{0}\|_{H^{-5/2}_{y}}\notag\\
&+\la^{-1} T\|\la^{-1/2}\dive (\Delta\ut+\nabla\dive \ut)\|_{L^{2}_{\tau}H^{-2}_{x}}.
\label{4.2.12}
\end{align}
From the estimate \eqref{2.5} we have that $\rt|\ut|^{2},\ \tilde{\pi}\in L^{\infty}_{t}L^{1}_{x}$, but $L^{1}$ is continuously embedded in $H^{-s_{0}}$, $s_{0}\geq 3/2$, so we have that $F_{2}\in L^{\infty}_{t}H^{-s_{0}-2}_{x}$. If we apply \eqref{s2} to $\st_{2}$ we obtain for any $s_{0}\geq 3/2$
\begin{align}
\|\st_{2}\|_{L^{4}_{\tau} W^{-s_{0}-2,4}_{y}}&+\|\partial_{\tau}\st_{2}\|_{L^{4}_{\tau} W^{-s_{0}-3,4}_{y}}\notag\\
&\lesssim 
\la^{-1/2}T\|\dive (\dive(\rt\ut\otimes\ut)+\nabla\tilde{\pi})\|_{L^{\infty}_{\tau}H^{-s_{0}-2}_{y}}.
\label{4.2.13}
\end{align}
By using the Poisson equation \eqref{4.2.2bis} we can rewrite $F_{3}$ as $F_{3}=\dive(\dive(\nabla \Vt\otimes\nabla \Vt)+\frac{1}{2}\nabla|\nabla \Vt|^{2})$. Taking into account \eqref{2.5}, as for $F_{2}$, we get $F_{3}\in L^{\infty}_{\tau}H^{-s_{0}-2}_{x}$, for any $s_{0}\geq 3/2$. Hence $\st_{3}$ satisfies
\begin{align}
\|\st_{3}\|_{L^{4}_{\tau} W^{-s_{0}-2,4}_{y}}&+\|\partial_{\tau}\st_{3}\|_{L^{4}_{\tau} W^{-s_{0}-3,4}_{y}}\notag\\
&\lesssim \la^{-1/2} T\|\dive(\nabla \Vt\otimes\nabla \Vt)+\frac{1}{2}\nabla|\nabla \Vt|^{2}\|_{L^{\infty}_{\tau}H^{-s_{0}-2}_{y}}.
\label{4.2.14}
\end{align}
Summing up \eqref{4.2.12}, \eqref{4.2.13}, \eqref{4.2.14}, $\st$ verifies
\begin{align}
\|\st\|_{L^{4}_{\tau} W^{-s_{0}-2,4}_{x}}+\|\partial_{\tau}\se\|_{L^{4}_{t} W^{-s_{0}-3,4}_{y}}&\lesssim 
\|\st_{0}\|_{H^{-3/2}_{y}}+\|\partial_{\tau}\st_{0}\|_{H^{-5/2}_{y}}\notag\\
&+\la^{-1} \|\la^{-1/2}\dive (\Delta\ut+\nabla\dive \ut)\|_{L^{2}_{\tau}H^{-2}_{y}}\notag\\
&+\la^{-1}T\|\dive (\dive(\rt\ut\otimes\ut)+\nabla\tilde{\pi})\|_{L^{\infty}_{\tau}H^{-s_{0}-2}_{y}}\notag\\
&+\la^{-1} T\|\dive(\st\nabla\Vt)\|_{L^{\infty}_{\tau}H^{-s_{0}-1}_{y}}\label{4.2.111}
\end{align}
Finally,  since 
\begin{equation*}
\|\st\|_{L^{q}_{\tau}W^{k,p}_{y}}=\la^{-\frac{1}{q}+k-\frac{3}{p}}\|\se\|_{L^{p}([0,T];L^{q}(\R^{3}))}
\end{equation*}
and by using \eqref{4.1.7}, we end up with \eqref{4.2.11}.
\end{proof}
\section{Strong Convergence of the velocity field}
In this section we will study the strong convergence of the velocity field $\ul$. This will be achieved by studying separately the convergence of the divergence free vector field $P\ul$ and of the gradient vector field $Q\ul$.
\subsection{Strong convergence of $\mathbf{Q\ul}$}
Here we prove the convergence of $Q\ul$ to $0$. The main tool in this process lies on the fact that $Q\ul$ can be computed in terms of $\sls$, so we can  use the estimate \eqref{4.2.11} combined with the Young type inequalities \eqref{y1}, \eqref{y2}. 
\begin{proposition}
Let us consider the solution $(\rl, \ul, V^{\la})$ of the Cauchy problem for the system \eqref{3.1}. Assume that the hypotheses (ID) hold. Then  as $\la\downarrow 0$,
\begin{equation}
Q\ul \longrightarrow 0 \quad \text{strongly in $ L^{2}([0,T];L^{p}(\R^{3}))$ for any $p\in [4,6)$ }.
\label{5.1.2}
\end{equation}
\label{p5.1.1}
\end{proposition}
\begin{proof}
In order to prove the Proposition \ref{p5.1.1} we split $Q\ul$ as follows
\begin{equation*}
\|Q\ul\|_{L^{2}_{t}L^{p}_{x}}\leq \|Q\ul-Q\ul\ast j_{\alpha}\|_{L^{2}_{t}L^{p}_{x}}+\|Q\ul\ast j_{\alpha}\|_{L^{2}_{t}L^{p}_{x}}=J_{1}+J_{2},
\end{equation*}
where $j_{\alpha}$ is the smoothing kernel defined in Lemma \ref{ly}.
Now we estimate separately $J_{1}$ and $J_{2}$.
For $J_{1}$ by using \eqref{y1} we get
\begin{equation}
\label{5.1.3}
J_{1}\leq \alpha^{1-3\left(\frac{1}{2}-\frac{1}{p}\right)}\|\nabla \ul\|_{L^{2}_{t,x}}.
\end{equation}
To estimate $J_{2}$ we take into account the definition \eqref{3.1.1} and so we split $J_{2}$ as
\begin{equation}
J_{2}\leq \|Q(\sls\ul)\ast j_{\alpha}\|_{L^{2}_{t}L^{p}_{x}}+\|Q(\rl\ul)\ast j_{\alpha}\|_{L^{2}_{t}L^{p}_{x}}=J_{2,1}+J_{2,2}.
\label{5.1.4}
\end{equation}
For $J_{2,1}$ we use  \eqref{4.1.3} and \eqref{y2}, so we have
\begin{equation}
J_{2,1}\leq\alpha^{-1-3\left(\frac{1}{2}-\frac{1}{p}\right)}\|\sls \ul\|_{L^{2}_{t}H^{-1}_{x}}.
\label{5.1.5}
\end{equation}
From the identity $Q(\sls\ul)=\nabla\Delta^{-1}\partial_{t}\sls$ and by the inequality \eqref{y2} we get  $J_{2,2}$ satisfies the following estimate
\begin{align}
J_{2}&= \la^{1/2}\|\la^{-1/2}\nabla\Delta^{-1}\partial_{t}\sls\ast j\|_{L^{2}_{t}L^{p}_{x}}\notag\\
&\leq \la^{1/2}\alpha^{-s_{0}-4-3\left(\frac{1}{4}-\frac{1}{p}\right)}\|\la^{-1/2}\partial_{t}\se\|_{L^{2}_{t}W^{-s_{0}-4,4}_{x}}\notag\\
&\leq \la^{1/2}\alpha^{-s_{0}-4-3\left(\frac{1}{4}-\frac{1}{p}\right)}T^{1/2}\|\la^{-1/2}\partial_{t}\se\|_{L^{4}_{t}W^{-s_{0}-4,4}_{x}}.
\label{5.1.6}
\end{align}
Now, summing up \eqref{5.1.4}, \eqref{5.1.5} and \eqref{5.1.6} we get
\begin{equation}
\|Q\ue\|_{L^{2}_{t}L^{p}_{x}}\lesssim  \alpha^{1-3\left(\frac{1}{2}-\frac{1}{p}\right)}+C_{T}\la^{1/2}\alpha^{-s_{0}-4-3\left(\frac{1}{4}-\frac{1}{p}\right)},
\label{5.1.7}
\end{equation}
Finally,  we choose $\alpha$ in terms of $\la$, for example in a way that the two terms on the right-hand side of the inequality \eqref{5.1.7} are of the same order, namely
\begin{equation}
\alpha=\la^{\frac{2}{17+4s_{0}}}.
\end{equation}
Therefore, we obtain
\begin{equation*}
\displaystyle{\|Q\ul\|_{L^{2}_{t}L^{p}_{x}}\leq C_{T}\la^{ \frac{6-p}{p(17+4s_{0})}}\quad \text{for any $p\in [4,6)$.}}
\end{equation*}
\end{proof}
\subsection{Strong convergence of $\mathbf{P\ul}$}
It remains to prove the strong compactness of the incompressible component of the velocity field.  To achieve this goal we need to recall here, The compactness can be obtained by looking at some time regularity properties of $P\ul$ and by using the Theorem \ref{S}, but before we need to prove  the following lemma.
\begin{lemma}
\label{ll1}
Let us consider the solution $(\rl, \ul, \vl)$ of the Cauchy problem for the system \eqref{3.1}. Assume that the hypotheses (ID) hold. Then for all $h\in(0,1)$, we have
\begin{equation}
\label{5.2.1}
\|P\ul(t+h)-P\ul(t)\|_{L^{2}([0,T]\times \R^{3})}\leq C_{T}h^{2/5}.
\end{equation}
\end{lemma}
\begin{proof}
Let us set $z^{\la}=\ul(t+h)-\ul(t)$, we have
\begin{align}
\hspace{-0,25 cm}\|P\ul(t+h)-P\ul(t)\|^{2}_{L^{2}_{t,x}}&=\int_{0}^{T}\!\!\int_{\R^{3}}dtdx(Pz^{\la})\cdot(Pz^{\la}-Pz^{\la}\ast j_{\alpha})\notag\\&+\int_{0}^{T}\!\!\int_{\R^{3}}dtdx(Pz^{\la})\cdot(Pz^{\la}\ast j_{\alpha})=I_{1}+I_{2}.
\label{5.2.2}
\end{align}
By using \eqref{y1} together with \eqref{4.1.2} we can estimate $I_{1}$ in the following way
\begin{equation}
I_{1}\leq \|Pz^{\la}\|_{L^{2}_{t,x}}\|Pz^{\la}(t)-(Pz^{\la}\ast j_{\alpha})(t)\|_{L^{2}}\lesssim \alpha \|\ul\|_{L^{2}_{t,x}}\|\nabla\ul\|_{L^{2}_{t,x}}.
\label{5.2.3}
\end{equation}
In order to estimate $I_{2}$ we split it as follows
\begin{align}
I_{2}&=\int_{0}^{T}\!\!\int_{\R^{3}}dtdxP(\rl z^{\la})\cdot(Pz^{\la}\ast j_{\alpha})+\int_{0}^{T}\!\!\int_{\R^{3}}dtdxP(\sls z^{\la})\cdot(Pz^{\la}\ast j_{\alpha})\notag\\
&=I_{2,1}+I_{2,2}.
\label{5.2.4}
\end{align}
$I_{2,2}$ can be estimated by taking into account \eqref{4.1.2}, \eqref{4.1.3} and $\eqref{3.1}_{3}$
so we have
\begin{align}
I_{2,2}=&\la^{2}\int_{0}^{T}\int_{\R^{3}}dtdx (\Delta\vl z^{\la}(Pz^{\la}\ast j_{\alpha})\notag\\
&=\la^{2}\int_{0}^{T}\int_{\R^{3}}dtdx \left[(\nabla \vl z^{\la})(\nabla Pz^{\la}\ast j_{\alpha})+\nabla \vl\nabla z^{\la}(Pz^{\la}\ast j_{\alpha})\right]\notag\\
&\leq \la\|\la \nabla \vl z^{\la}+ \la \nabla \vl \nabla \ul\|_{L^{2}_{t}L^{1}_{x}}\|\nabla Pz^{\la}\ast j_{\alpha}\|
\notag\\
&\leq \la \alpha^{-3/2}\|\nabla \ue\|_{L^{2}_{t,x}}\|\la \nabla \vl z^{\la}+ \la \nabla \vl \nabla\ul\|_{L^{2}_{t}L^{1}_{x}}.
\label{5.2.5}
\end{align}
Now we estimate $I_{2,1}$. Let us reformulate $P(\rl z^{\la})$ in integral form by using the equation $\eqref{3.1}_{2}$ and the Poisson equation $\eqref{3.1}_{3}$, hence
\begin{align}
I_{2,1}&\leq\left|\int_{0}^{T}dt \int_{\R^{3}}dx \int_{t}^{t+h}ds(\dive(\rl\ul\otimes\ul)+\Delta \ul)(s,x)\cdot (Pz^{\la}\ast j_{\alpha})(t,x)\right|\notag\\
&+\left|\int_{0}^{T}dt \int_{\R^{3}}dx \int_{t}^{t+h}dsP\left(\frac{\sls}{\la}\nabla V^{\la}\right)(s,x)\cdot (Pz^{\la}\ast j_{\alpha})(t,x)\right|\notag\\
&=\left|\int_{0}^{T}dt \int_{\R^{3}}dx \int_{t}^{t+h}ds(\dive(\rl\ul\otimes\ul)+\Delta \ul)\cdot (Pz^{\la}\ast j_{\alpha})(t,x)\right|\notag\\
&+\left|\int_{0}^{T}dt \int_{\R^{3}}dx \int_{t}^{t+h}ds\la^{2}\dive(\nabla V^{\la}\otimes\nabla V^{\la}))(s,x)\cdot (Pz^{\la}\ast j_{\alpha})(t,x)\right|.
\label{5.2.6}
\end{align}
Then, by integrating by parts, by using \eqref{y2} with $s=0$, $p=\infty$, $q=2$, we deduce
\begin{align}
I_{2,1}&\leq  
 h\|\nabla\ul\|^{2}_{L^{2}_{t,x}}
\notag\\
&+C\alpha^{-3/2}T^{1/2}h\|\nabla\ul\|_{L^{2}_{t,x}}\left(\|\rl |\ul |^{2}\|_{L^{\infty}_{t}L^{1}_{x}}+\||\la\nabla V^{\la}|^{2}\|_{L^{\infty}_{t}L^{1}_{x}}\right).
\label{5.2.7}
\end{align}
Summing up $I_{1}$, $I_{2,1}$, $I_{2,2}$ and by taking into account \eqref{2.5} we have
\begin{equation*}
\|P\ul(t+h)-P\ul(t)\|^{2}_{L^{2}([0,T]\times \R^{3})}\leq C(h^{2/5}+h+h\alpha^{-3/2}T^{1/2}+2\alpha^{-3/2}\la),
\label{5.2.8}
\end{equation*}
by choosing $ \lambda< \alpha$ and $\alpha=h^{2/5}$, we end up with \eqref{5.2.1}.
\end{proof}
\begin{corollary}
Let us consider the solution $(\rl, \ul, \vl)$ of the Cauchy problem for the system \eqref{3.1}. Assume that the hypotheses (ID) hold. Then  as $\la\downarrow 0$
\label{c3}
\begin{equation}
P\ul \longrightarrow Pu, \qquad \text{strongly in $L^{2}(0,T;L^{2}_{loc}(\R^{3}))$}.
\label{5.2.9}
\end{equation}
\end{corollary}
\begin{proof}
By using the Lemma \ref{ll1} and the Theorem \ref{S} and the Proposition \ref{p5.1.1} we get \eqref{5.2.8}. 
\end{proof}

\section{Convergence of the electric field}

This section is addressed to the study of the convergence of the electric field $\el=\nabla\vl$. By the a priori estimate \eqref{4.1.3b} we only know that  $\la \el$ is bounded in $L^{\infty}_{t}L^{2}_{x}$ which does not give enough information to pass into the limit in the quadratic term $\rl\nabla\vl=\dive(\la\el\otimes \la\el)-1/2\nabla|\la\el |^{2}$, appearing in the righthand side of $\eqref{3.1}_{2}$. Hence the problem is how to recover the weak continuity of quadratic forms in $L^{2}$. Since $\la\el$ is bounded in $L^{\infty}_{t}L^{2}_{x}$ we can define the so called microlocal defect measure introduced by P. G\`erard in \cite{Ger91} and by L. Tartar in \cite{Tar90} (H-measures), but in order to handle time oscillations we need to introduce correctors. In this section we will be able to prove the following theorem.
\begin{theorem}
\label{t6}
Let be $(\rl,\ul,\el)$ a sequence of solutions of the Navier Stokes Poisson system \eqref{3.1}, then
\begin{itemize}
\item[i)] there exists $E^{+}$, $E^{-}$ in $L^{\infty}((0,T),L^{2}(\R^{3}))$,
\item[ii)] there exists a positive measure $\nu^{E}$ on $\R^{3}\times S^{2}$ depending measurably on $t$
\end{itemize}
such that for all pseudodifferential operators $A\in \psi_{comp}^{0}(\R^{3},\mathcal{K}(\R^{3}))$, and of symbol $a(x,\xi)$ and for all $\phi\in \mathcal{D}(0,t)$ one has
\begin{align}
\lim_{\la\to 0}\int dt\phi(t) \la^{2}(A\el,\el)&=\int dt\phi(t) (AE^{+},E^{+})+\int dt\phi(t) (AE^{-},E^{-})\notag\\
&+\int dt\phi(t) \!\!\int_{\R^{3}\times S^{2}}tr\left(a(x,\xi)\frac{\xi\otimes \xi}{|\xi|^{2}}\right)d\nu^{E}.
\label{6.1a}
\end{align}
\end{theorem}
First we rewrite \eqref{4.2.3} in terms of $\el$, namely
\begin{align}
\la^{2}\partial_{tt}\el+\el&=\dive\Delta^{-1}\nabla\dive\left(\rl\ul\otimes\ul +(\rl)^{\gamma}\mathbb{I}-\la^{2}\el\otimes\el\right)\notag\\
&+\frac{\la^{2}}{2}\dive\left(|\el |^{2}\mathbb{I}\right)-2\nabla\dive \ul=F^{\la},
\label{E}
\end{align}
then we observe that by using  \eqref{4.2.11} and the uniqueness of the weak limit we have
\begin{equation}
\label{6.1}
\la\nabla\vl \rightharpoonup 0 \quad \text{weakly in $L^{2}(0,T;L^{2}(\R^{3}))$.}
\end{equation}
By \eqref{6.1} we see that we are exactly in the framework described by P. G\`erard, but we have to pay attention to one fact. In our case in the quadratic form $\la^{2}\langle A\el,\el \rangle$, $A$ is a pseudodifferential operator homogenous only with respect to the $x$ variable and in the general case we cannot extend it to a pseudodifferential operator homogenous in $(x,t)$. Hence we have to work on $\la \el$ in order to isolate the components that oscillates fast in time, for that reason we introduce what we call the correctors of the electric field.
By using \eqref{E} and Duhamel's formula we can write the electric field $\el$ as
\begin{align}
\el(t,x)&=\int_{0}^{t}\frac{F^{\la}(s,x)}{2i\la}\left(e^{i\frac{t-s}{\la}}-e^{-i\frac{t-s}{\la}}\right)ds\notag\\
&+\frac{\mathcal{E}^{\la}_{1}(x)}{\la}e^{it/\la}+\frac{\mathcal{E}^{\la}_{2}(x)}{\la}e^{-it/\la},
\label{6.2}
\end{align}
where $\mathcal{E}^{\la}_{1}$ and $\mathcal{E}^{\la}_{2}$ are two functions in $L^{2}_{x}$ defined by the initial data of $\el$.
In order to understand how to isolate the oscillating terms let us consider the equation \eqref{E} in the case when $F^{\la}$ does not depend on $x$ and $\mathcal{E}^{\la}_{1}=\mathcal{E}^{\la}_{2}=0$. Then, if we take the Fourier transform with respect to time we have ($\hat{E}$ denotes the Fourier transform with respect to time)
$$\la\hat{\el}=\frac{\la}{1-\la^{2}|\tau|^{2}}\hat{F^{\la}},$$ 
we can see that all the $L^{2}$-mass of $\la\el$ is concentrated in $\tau=\pm 1/\la$ as $\la\to 0$. This simple facts leads us  to introduce correctors in time of  order $1/\la$. So we define
\begin{equation}
\label{6.3}
\el_{+}=\la e^{-it/\la}\el \qquad \el_{-}=\la e^{it/\la}\el
\end{equation}
In particular they take into account of the $L^{2}$-mass of $\la\el$ around $1/\la$.
By construction it easily follows that $\el_{+}$ and $\el_{-}$ are bounded in $L^{2}_{t,x}$ and converge weakly to $E^{+}$ and $E^{-}$ respectively. Moreover we have
\begin{lemma}
\label{l6.2}
Let be $(\rl,\ul,\el)$  a sequence of solutions of the Navier Stokes Poisson system \eqref{3.1} which satisfy (ID), then one has
$$\mathcal{E}^{\la}_{1}(x)+\int_{0}^{T}ds\frac{F^{\la}(s,x)}{2i}e^{-is/\la}\rightharpoonup E^{+}\quad \text{in $\mathcal{D}((0,T)\times \R^{3})$}.$$
The same holds for $E^{-}$.
\end{lemma}
\begin{proof}
The proof follows by using \eqref{6.2} and Proposition \ref{p4.1}.
\end{proof}
So, if we look at the limit of $\la\el-e^{it/\la}E^{+}-e^{-it/\la}E^{-}$ as $\la\to 0$, we expect to take away the $L^{2}$-mass of $\la\el$ which concentrates around $1/\la$. Now we can define
\begin{equation}
\widetilde{\el}=\el-e^{it/\la}\frac{E^{+}}{\la}-e^{-it/\la}\frac{E^{-}}{\la},
\label{6.3a}
\end{equation}
then we can prove the following lemma.
\begin{lemma}
Let be $(\rl,\ul,\el)$ be a sequence of solutions of the Navier Stokes Poisson system \eqref{3.1} which satisfy (ID), then it holds
$$\la \widetilde{\el}\rightharpoonup 0 \qquad \text{weakly in $L^{2}(0,T, L^{2}(\R^{3}))$}.$$
\end{lemma}
\begin{proof}
The proof follows by taking into account \eqref{6.1} and  that $\la\widetilde{\el}$ is bounded in $L^{2}_{t,x}$.
\end{proof}
At this point we can hope that the weak convergence of $\la\widetilde{\el}$ is caused only by spatial oscillations, which allow us to introduce the microlocal defect measure in space. In order to do this, since the solutions are defined only in $(0,T)$, we need to extend $\el$ and $F^{\la}$ to $0$ out of this interval and to cut-off the frequencies greater than a certain quantity. This will be done in the next proposition.
\begin{proposition}
With the same assumption as in Lemma \ref{l6.2} we have
\begin{equation}
\int_{|\xi|\leq R}dx\int_{\R}dt |\la \mathcal{F}\widetilde{\el}(t,x)|\rightarrow 0,
\label{6.3b}
\end{equation}
for any $R$ independent on $\la$.
\end{proposition}
\begin{proof}
Let be $\chi_{R}$ the characteristic function on $B(0,R)$ and let be $\mathcal{T}_{R}=\mathcal{F}^{-1}\chi_{R}\mathcal{F}$ the operator that cuts the frequencies greater than $R$, clearly $\mathcal{T}_{R}$ is a bounded operator from $L^{2}$ to $H^{s}$, for any $s\geq 0$ and $\nabla\mathcal{T}_{R}$=$\mathcal{T}_{R}\nabla$.
If we apply $\mathcal{T}_{R}$ to \eqref{6.2} we have
\begin{align}
\mathcal{T}_{R}\el(t,x)&=\int_{0}^{t}\frac{\mathcal{T}_{R}F^{\la}(s,x)}{2i\la}\sin\left(\frac{t-s}{\la}\right)ds\notag\\
&+\frac{\mathcal{T}_{R}\el_{1}(x)}{\la}e^{it/\la}+\frac{\mathcal{T}_{R}\el_{2}(x)}{\la}e^{-it/\la}
\label{6.4}
\end{align}
and by the estimates of Proposition \ref{p4.1} we have that $\mathcal{T}_{R}F^{\la}$, $\mathcal{T}_{R}\el_{1}$, $\mathcal{T}_{R}\el_{2}$ are bounded in $L^{\infty}_{t}L^{2}_{x}$.
Since the solutions that we consider are defined in the time interval $0\leq t\leq T$, in order to use the Fourier transform in time we need to extend them to 0, so we get that $\mathcal{T}_{R}F^{\la}$ is bounded in $L^{2}_{t,x}$.
Let us introduce 
\begin{equation}
\label{6.5}
H^{\la}_{+}=\int_{0}^{t}\mathcal{T}_{R}F^{\la}(s,x)e^{-is/\la}ds+2i\mathcal{T}_{R}\el_{1}(x).
\end{equation}
If we compute the space and time Fourier transform of \eqref{6.5}, for $\tau$ large enough we get
\begin{equation}
\label{6.6}
\mathcal{F}_{t,x}H^{\la}_{+}(\tau,\xi)=\frac{1}{\tau}\mathcal{F}_{t,x}\mathcal{T}_{R}F^{\la}\left(\tau+\frac{1}{\la}, \xi\right),
\end{equation}
so $\mathcal{F}_{t,x}H^{\la}_{+}$ is in $L^{2}$ in a neighborhood of $|\tau|=\infty$, moreover we have that
\begin{equation}
\int_{\R^{3}}\int_{|\tau|\geq A}|\mathcal{F}_{t,x}H^{\la}_{+}(\tau,\xi)|^{2}d\tau\leq \frac{1}{A^{2}}\int_{\R^{3}}\int_{\R}|\mathcal{T}_{R}F^{\la}(t,x)|^{2}dxdt.
\label{6.7}
\end{equation}
As a consequence we have that all the mass of $\la\mathcal{F}_{t,x}\mathcal{T}_{R}\el$ is concentrated in $1/\la$. In fact 
\begin{equation}
\la\mathcal{F}_{t,x}\mathcal{T}_{R}\el=\frac{1}{2i}\left[\mathcal{F}_{t,x}H^{\la}_{+}\left(\tau+\frac{1}{\la},\xi\right)-\mathcal{F}_{t,x}H^{\la}_{-}\left(\tau-\frac{1}{\la},\xi\right)\right],
\end{equation}
we have that for any $\eta>0$ there exists $A$ such that
$$\int_{\R^{3}_{\xi}}\int_{\left|\tau \pm\frac{1}{\la}\right|\geq A}|\la\mathcal{F}_{t,x}\mathcal{T}_{R}\el|^{2}d\xi d\tau\leq \eta.$$
In order to proof  \eqref{6.3b} we  take into account the decomposition of $\widetilde{\el}$ in \eqref{6.3a} and the following properties.
For any $\eta>0$ there exists $A$ such that for any $\la<1$ one has
\begin{equation}
\int_{\R^{3}_{\xi}}\int_{\left|\tau \pm\frac{1}{\la}\right|\geq A}|\la\mathcal{F}_{t,x}e^{\pm it/\la}E^{\pm}|^{2}d\xi d\tau\leq \eta.
\label{6.8}
\end{equation}
Moreover 
\begin{align}
&\int_{\R^{3}_{\xi}}\int_{\left|\tau -\frac{1}{\la}\right|\geq A}|\mathcal{F}_{t,x}(\la \mathcal{T}_{R}\widetilde{\el}-e^{it/\la}E^{+})|d\tau d\xi\notag\\
&= \int_{\R^{3}_{\xi}}\int_{\tau|\geq A}|\mathcal{F}_{t,x}((\la \mathcal{T}_{R}\widetilde{\el}e^{-it/\la}-\mathcal{T}_{R} E^{+})|^{2}d\xi d\tau.
\label{6.9}
\end{align}
and so for $E^{-}$.
On the other hand we know that
\begin{equation}
\la e^{-it/\la}\mathcal{T}_{R}\widetilde{\el}=\la e^{-it/\la}\mathcal{T}_{R}\el - \mathcal{T}_{R}E^{+} -e^{-2it/\la}\mathcal{T}_{R}E^{-},
\label{6.10}
\end{equation}
the same holds for $E^{-}$.
By taking into account \eqref{6.8}, \eqref{6.9}, \eqref{6.10} and Parseval's identity we conclude in the following way.
\begin{align}
\int_{|\xi|\leq R} d\xi \int dt |\la \mathcal{F} \widetilde{\el}|^{2}&=\int_{|\xi|\leq R} d\xi \int d\tau |\la \mathcal{F}_{t,x} \mathcal{F}^{-1}\chi_{R}\mathcal{F} \widetilde{\el}|^{2}\notag\\
&=\int_{|\xi|\leq R} d\xi \int d\tau |\mathcal{F}_{t,x} \mathcal{T}_{R} \widetilde{\el}|^{2} \rightarrow 0.
\end{align}
\end{proof}
Now  we are ready to prove the existence of a microlocal defect measure for the electric field $\el$. We start by proving the $L^{2}$-orthogonality of $E^{+}$, $E^{-}$ and $ \widetilde{\el}$.

\begin{proposition}
\label{p6.2}
For any $A\in \psi_{comp}^{0}(\R^{3},\mathcal{K}(\R^{3}))$ and for any $\phi\in \mathcal{D}(0,T)$ it holds
\begin{align}
\lim_{\la\to 0} \la^{2}\int  dt \phi(t)(A \el, \el)&= \lim_{\la\to 0} \la^{2}\int  dt \phi(t)(A \widetilde{\el}, \widetilde{\el})\notag\\
&+ \int  dt \phi(t)(AE^{+}, E^{+}) +\int  dt \phi(t)(AE^{-}, E^{-}).
\label{6.11}
\end{align}
\end{proposition}
\begin{proof}
First of all we observe that 
\begin{equation}
\lim_{\la\to 0}\int dt\phi(t)(A E^{+}e^{it/\la}, E^{-}e^{-it/\la})= 0
\label{6.12}
\end{equation}
Then we also have 
\begin{equation}
\lim_{\la\to 0}\int dt\phi(t)\la(A\widetilde{\el}, E^{+}e^{-it/\la})= 0
\label{6.13}
\end{equation}
In fact if we denote by $A^{\ast}$ the adjoint operator of $A$ we have that $A^{\ast} E^{+}$ is bounded in $L^{2}$ and as a consequence for any $\eta>0$, there exists $B>0$ such that
\begin{equation}
\label{6.14}
\int_{|\xi|\geq B}\int dt |\mathcal{F} A^{\ast} E^{+}|^{2}dt \leq \eta
\end{equation}
Combining \eqref{6.14} with \eqref{6.3b} we get \eqref{6.13}.
\end{proof}
In order to prove the Theorem \ref{t6} and to get \eqref{6.1a} it remains  only to investigate  
\begin{equation}
\label{6.15} \lim_{\la\to 0} \int  dt \la^{2}\phi(t)(A \widetilde{\el}, \widetilde{\el}).
\end{equation}
The sequence $\la \widetilde{\el}$ fits in the framework of microlocal defect measures of P. G\`erard on \cite{Ger91} but as already explained we need  his prove to our sequence.
\begin{proposition}
Let $\wl$ be  a bounded sequence of functions of $L^{2}_{t,x}$ which converges weakly to $0$, such that for every compact set $K\subset \R^{3}$ one has
$$\lim_{\la\to 0}\int_{K}d\xi\int dt |\mathcal{F}\wl(t,\xi)|=0,$$
then, there exists a positive measure $\nu^{GT}$ on $\R\times\R^{3}\times S^{2}$, such that for any $A\in \psi_{comp}^{0}(\R^{3},\mathcal{K}(\R^{3}))$ and of principal symbol $a(x,\xi)$ and for any $\phi\in \mathcal{D}(\R)$ it holds
$$\lim_{\la\to 0}\int dt\phi(t)(A\wl,\wl)=\langle \nu^{GT}(dt,dx,d\xi), \phi(t)a(x,\xi)\rangle.$$
\label{pg}
\end{proposition}
To prove the previous theorem we follow the same line of arguments as in \cite{Ger91}, we start with the following lemma.
\begin{lemma}
With the assumptions as in the Propositions \ref{pg} it holds
\begin{equation}
\label{6.16}
\lim_{\la\to 0}\Im\int dt\phi(t)(A\wl,\wl)=0,
\end{equation}
\begin{equation}
\label{6.17}
\lim_{\la\to 0}\Re\int dt\phi(t)(A\wl,\wl)\geq 0.
\end{equation}
\end{lemma}
\begin{proof}
Since $A$ is Hermitian we  have
$$\Im(A\wl,\wl)=\frac{1}{2i}((A-A^{\ast})\wl,\wl),$$
where $A-A^{\ast}\in \psi_{comp}^{-1}(\R^{3},\mathcal{K}(\R^{3}))$, so \eqref{6.16} follows by using Proposition \eqref{p2.2}. For the real part let be $\delta>0$, then $a+\delta\in C^{\infty}(S^{\ast}\Omega, \mathcal{L}(\R^{3})$ and we can extract the square root $B$, namely $b=\delta^{1/2}+b'$. Let be $B'$ such that $B'=OPS(b')$ and $B$ such that $B=\varphi(\delta^{1/2}+B')$, $\varphi\in C^{\infty}_{0}(\R^{3})$, $\varphi=1$ on the support of $a$, then $OPS^{0}(B)=OPS^{0}(\varphi \delta^{1/2}+\varphi B')=\varphi b\in \psi_{comp}^{0}(\R^{3},\mathcal{L}(\R^{3}))$. So we have that
$$B^{\ast}B=|\varphi|^{2}\delta+A+R, \qquad R\in \psi_{comp}^{-1}(\R^{3},\mathcal{K}(\R^{3})),$$
but then we have
$$\Re\int dt\phi(t)(A\wl,\wl)\geq -\delta \|\varphi\wl\|^{2}_{L^{2}}+ \Re(R\wl,\wl).$$
We end up with \eqref{6.17} by sending $\la$ to $0$ and by using again Proposition \eqref{p2.2}.
\end{proof}
From now on the proof of Proposition \ref{pg} follows the same line of arguments as in \cite{Ger91}.
So we can apply the Propositions \ref{pg} to the sequence $\la\widetilde{\el}$ end we can conclude that there exists a positive measure $\nu^{\widetilde{E}}$ such that
$$\lim_{\la\to 0}\int dt\phi(t)(A\la\widetilde{\el},\la\widetilde{\el})=\langle \nu^{\widetilde{E}}(dt,dx,d\xi), \phi(t)a(x,\xi)\rangle.$$
If we apply the remark in Exercise 1.5 of \cite{Ger91} , since $\la\widetilde{E^{\la}}$ is a  gradient and we are in the finite dimensional case we have that there exists a positive measure $\nu^{E}$ such that
$$\nu^{\widetilde{\el}}(dt,dx,d\xi)=\xi_{i}\xi_{j}\nu^{E}(dt,dx),$$
this ends the proof of the Theorem \ref{t6}.

\section{Proof of the Main Theorem \ref{tm1}}
\begin{itemize}
\item[{\bf(i)}] It follows from \eqref{4.1.1} and \eqref{4.1.3b}.
\item[{\bf(ii)}] It follows from \eqref{4.1.2}.
\item[{\bf(iii)}] It is proved in Proposition \ref{p5.1.1}.
\item[{\bf(iv)}] By taking into account that we can decompose $\ul$ as $\ul=P\ul+Q\ul$ and by using Proposition \ref{p5.1.1} with Corollary \ref{c3} we get 
\begin{equation*}
P\ul\longrightarrow u\quad \text{strongly  in $L^{2}([0,T];L^{2}_{loc}(\R^{3}))$}.
\end{equation*}
\item[{\bf(v)}] It follows from Theorem \ref{t6}.
\item[{\bf(vi)}] First of all we apply the Leray projector $P$ to the momentum equation of the system \eqref{3.1}, then we have
\begin{equation}
\partial_{t}P(\rho^{\lambda} u^{\lambda})+P\dive(\rho^{\lambda} u^{\lambda}\otimes u^{\lambda})=\mu\Delta Pu^{\lambda}+P\dive(\la\el\otimes\la\el).
\end{equation}
It is a straightforward computation to pass into the limit in the terms $\partial_{t}P(\rho^{\lambda} u^{\lambda})$ and $\Delta Pu^{\lambda}$, so, for any $\varphi\in \mathcal{D}([0,T]\times \R^{3})$ we obtain
\begin{equation}
\langle P(\partial_{t}(\rl\ul)-\mu\Delta \ul),\varphi\rangle\longrightarrow \langle P(\partial_{t}u-\mu\Delta u),\varphi\rangle.
\label{l1} 
\end{equation}
In order to study the convergence of  the convective term we decompose it in this way
\begin{align}
\langle P\dive(\rho^{\lambda} u^{\lambda}\otimes u^{\lambda}), \varphi\rangle&=
\langle P\dive((\rho^{\lambda} -1)u^{\lambda}\otimes u^{\lambda}), \varphi\rangle\notag\\
&+\langle P\dive( u^{\lambda}\otimes u^{\lambda}), \varphi\rangle\notag\\
&=I_{1}+I_{2}
\label{l2}
\end{align}
The term $I_{1}$ goes strongly to zero. In fact it is enough to take into account $\rl-1$ goes weakly to zero in $L^{\infty}_{t}L^{k}_{2}$, while by interpolation $\ul$ is strongly convergent in $L^{\infty}_{t}L^{k'}_{2}$. Concerning $I_{2}$ we have as $\la\to 0$,
\begin{align}
I_{2}&=\langle \dive(P\ul\otimes Q\ul),P\varphi\rangle\nonumber+\langle \dive(Q\ul\otimes Q\ul),P\varphi \rangle\nonumber\\&\longrightarrow \langle \dive(u \otimes u),P\varphi\rangle=\langle P\dive(u \otimes u), \varphi\rangle.
\label{l3}
\end{align}
Finally to establish the convergence  of the term $P\dive(\la\nabla\vl\otimes\la \nabla\vl)$  we have to take the limit of $\langle P\dive(\la\nabla\vl\otimes\la\nabla\vl), \varphi\rangle$ for any $\varphi\in \mathcal{D}([0,T]\times \R^{3})$, but we have no strong convergence for $\la\nabla\vl$ as $\la\to 0$, so we  apply the Theorem \ref{t6} and use the microlocal defect measure defined in \eqref{6.1a} and we have as $\la\to 0$,
\begin{align}
\label{l4}
\langle P\dive(\la\nabla\vl\otimes\la\nabla\vl),\varphi \rangle&=\langle \la\nabla\vl\otimes\la\nabla\vl, \nabla P\varphi \rangle\\
&+\int_{0}^{T}\int_{\R^{3}\times S^{2}}\nabla P\varphi \frac{\xi\otimes\xi}{|\xi|^{2}}d\nu^{E}dxdt\notag\\
&+\int_{0}^{T}\int_{\R^{3}}\nabla P\varphi(E^{+}\otimes E^{+}+E^{-}\otimes E^{-})dxdt.\notag
\end{align}
So, by using together \eqref{l1}, \eqref{l2}, \eqref{l3}, \eqref{l4}  we have that $u$ satisfies the following equation  in $\mathcal{D}'([0,T]\times \R^{3})$
\begin{align*}
P\Big(\partial_{t} u&-\Delta u+(u\cdot\nabla)u-\notag\\
& \dive(E^{+}\otimes E^{+}+E^{-}\otimes E^{-})-\dive\langle \nu^{E}, \frac{\xi\otimes \xi}{|\xi|^{2}}\rangle\Big)=0.
\end{align*}
\end{itemize}

\section{Equations for the correctors: proof of the Main Theorem \ref{tm2}}

The purpose of this section is to show  how to construct  the correctors $E^{+}$ and $E^{-}$ in the case of smooth solutions for the system \eqref{3.1}. In particular we will perform the quasineutral limit in the framework of the Theorem \ref{t2} and we will end up  with the proof of the Main Theorem \ref{tm2}. We will divide the proof of the Main Theorem \ref{tm2} in different steps. First of all we decompose the electric and the velocity fields in order to single out the oscillating parts, then we show the existence of the correctors and finally the equations that they satisfy. A similar analysis has been carried out in \cite{G95} in the case of a periodic domain for the Vlasov Poisson system.

For the computations we have to perform later on it is more convenient to  rewrite the expression \eqref{6.2} for $\el$ in terms of its Fourier transform,
\begin{align}
\FF\el(t,\xi)&=\int_{0}^{t}\frac{1}{\la}\FF F(s,\xi) \sin\left(\frac{t-s}{\la}\right)ds\notag\\
&+\FF \el(0,\xi)\cos\left(\frac{t}{\la}\right)+\la\FF \el_{t}(0,\xi)\sin\left(\frac{t}{\la}\right).
\label{7.1}
\end{align}
\subsection{Step 1: decomposition of the electric field}
In this section we decompose the electric field in a way to isolate the time oscillations.
First of all we define the following operator that cuts off the oscillations in time, for any $\phi\in C^{0}(0,T, H^{s})$, $s\geq 0$ we set
\begin{equation}
\mathcal{H}^{\la}\phi(t,x)=\frac{1}{2\pi\la}\int_{t}^{t+2\pi\la}\phi(\sigma, x)d\sigma, \qquad \mathcal{G}^{\la}=I-\mathcal{H}^{\la}.
\label{7.1.1}
\end{equation}
Then we decompose $\el$ in the following way
\begin{equation}
\el_{1}=\mathcal{G}^{\la}\el \qquad \el_{2}=\mathcal{H}^{\la}\el,
\end{equation}
clearly, $\el_{1}$ is the oscillatory part of $\el$, while $\el_{2}$ is its averaged part.
The following proposition hold
\begin{proposition}
\label{7.1}
Under the assumptions of the Theorem \ref{t2} there exist three vector fields (which are gradients), $\el_{1}$ , $\el_{2}$, and $W^{\la}$ such that $\el=\el_{1} +\el_{2}$ with
\begin{itemize}
\item[(i)] $\|\la\el_{1}\|_{L^{\infty}_{t}H^{s-1}_{x}}\leq C$,
\item[(ii)] $\partial_{t}W^{\la}=\el_{1}, \quad \|W^{\la}\|_{L^{\infty}_{t}H^{s-1}_{x}}\leq C$ and $W^{\la}\rightharpoonup 0$ weakly in $L^{2}$,
\item[(iii)]  $\|\el_{2}\|_{L^{\infty}_{t}H^{s-1}_{x}}\leq C$.
\end{itemize}
\end{proposition}
\begin{proof}
Since $F $ is uniformly bounded in $L^{\infty}_{t}H^{s-1}_{x}$ we have that   $\mathcal{H}^{\la}\el$ is bounded in $L^{\infty}_{t}H^{s-1}_{x}$ and so we get (i) and (iii). Now, if we define
\begin{align}
\FF W^{\la}&=\la \FF \el(0,\xi)\sin\left(\frac{t}{\la}\right)-\la \FF \el_{t}(0,\xi)\cos\left(\frac{t}{\la}\right)\notag\\
&+\int_{0}^{t}d\sigma\int_{0}^{\sigma}ds \frac{\la F(s,\xi)}{\la}\sin\left(\frac{\sigma-s}{\la}\right)+\int_{0}^{t}\FF \mathcal{H}^{\la}\el(\sigma, \xi)d\sigma.
\label{7.1.2}
\end{align}
 we easily obtain (ii).
 \end{proof}
 
\subsection{Step 2: decomposition and limit system for the velocity}
Now we decompose the velocity field $\ul$, in the following way
\begin{equation}
\ul=v^{\la}+W^{\la},
\label{7.2.1}
\end{equation}
where $W^{\la}$ is the corrector introduced in the Proposition \ref{7.1} and we can look at $v^{\la}$ as the velocity field $\ul$ without its oscillatory part $W^{\la}$. For $v^{\la}$ we can prove the following result.
\begin{proposition}
Let $v^{\la}=\ul-W^{\la}$, then for any $s'<s-2$ , $v^{\la}$ and $\rl$ converge in  $C^{0}(0,T;H^{s'}(\R^{3}))$, respectively to $v$ and $1$ and there exists a function $\Pi$, such that $v$ satisfies,
\begin{align}
\dive v&=0 \label{7.2.3}\\
\partial_{t}v+v\cdot\nabla v&-\Delta v=\nabla\Pi	
\label{7.2.2}       
\end{align}
\end{proposition}
\begin{proof}
First of all  we can observe that $v^{\la}$ is bounded in $L^{\infty}_{t}H^{s-1}_{x}$ and $\partial_{t}v^{\la}$ is bounded in $L^{\infty}_{t}H^{s-2}_{x}$, so we have that 
\begin{equation}
v^{\la}\longrightarrow v\quad \text{strongly in $C(0,T;H^{s'-1}_{loc}),$ for any $s'<s$.}
\end{equation}
Now we  rewrite the second equation of \eqref{3.1} in terms of $v^{\la}$ and $W^{\la}$.
\begin{align}
\rho\partial_{t}v^{\la}&+\rho(v^{\la}+W^{\la})\cdot\nabla (v^{\la}+W^{\la})+\nabla (\rho^{\la})^{\gamma}\notag\\&=\Delta (v^{\la}+W^{\la}) +\nabla\dive (v^{\la}+W^{\la})+ \rho\el_{2}.
\label{7.2.4}
\end{align}
By the Poisson equation $\la^{2}\dive \el=\rl -1$, we have that 
\begin{equation}
\rl\longrightarrow 1\quad \text{strongly in $C(0,T;H^{s-1})$,}
\label{7.2.4a}
\end{equation}
from this follows \eqref{7.2.3}. Then, by using Proposition \ref{7.1} we have that $W^{\la}$ and $\partial_{x}W^{\la}$ are bounded in $L^{2}$ and converges weakly to $0$, so we can pass into the limit in \eqref{7.2.4} and we conclude with \eqref{7.2.2}.
\end{proof}

\subsection{Step 3: Existence of the correctors}
In this section we will identify and establish the existence of the correctors.  First of all we introduce the operator $T^{\la}_{\pm}$, for any $\phi \in L^{\infty}_{t}L^{2}_{x}$ we set
\begin{equation}
\FF T^{\la}_{\pm}\phi(t,\xi)=e^{\mp t/ \la}\FF\phi(t,\xi), 
\end{equation}
By construction we have that $T^{\la}_{\pm}$ satisfies the following properties.
\begin{itemize}
\item [(T1)] $T^{\la}_{\pm}$ are selfadjoint, act isometrically on $L^{\infty}_{t}H^{s}_{x}$, for all $s$ and $T^{\la}_{+} T^{\la}_{-}=T^{\la}_{-}T^{\la}_{+} =I$,
\item [(T2)] if $\phi^{\la}\rightarrow \phi$, strongly in $L^{2}$, then, $T^{\la}_{+} \rightharpoonup 0$, weakly in $L^{2}$,
\item [(T3)] if $\phi, \psi \in L^{\infty}_{t}H^{s}_{x}$, $s>d/2$, $T^{\la}_{+}\phi T^{\la}_{-}\psi \rightharpoonup 0$ weakly in $L^{2}_{t,x}$.
\end{itemize}
In the next proposition we prove the existence of corrector for the electric field $\el$ and the velocity field $\ul$.
\begin{proposition}
\label{p7.3}
There exists two functions $E^{+}$ and $E^{-}$ in $C^{0}(H^{s-1}_{loc})$ such that for all $s'<s$,
\begin{itemize}
\item [(c1)] $\|\la \el_{1}-T^{\la}_{-}E^{+}-T^{\la}_{+}E^{-}\|_{C^{0}(H^{s'-1}_{loc})}\rightarrow 0,$
\item [(c2)] $\|W^{\la}-\frac{1}{i}T^{\la}_{-}E^{+}-\frac{1}{i}T^{\la}_{+}E^{-}\|_{C^{0}(H^{s'-1}_{loc})}\rightarrow 0.$
\end{itemize}
\end{proposition}
\begin{proof}
We can split $\el$ in two components $\el_{+}$ and $\el_{-}$, in the following way
\begin{equation}
\hspace{-0.3cm}\FF\el_{+}(t,\xi)=\la e^{it/\la}\!\left(\frac{\FF \el(0,\xi)}{2}+\frac{\la \FF \el_{t}(0,\xi)}{2i}+\!\!
\int_{0}^{t} \!\frac{\FF F(s,\xi)}{2i\la}e^{-it/\la}ds\right)
\label{7.3.2}
\end{equation}
and we define in a  similar way  $\el_{-}$. We can easily verify that
\begin{equation}
\|T_{\pm}E^{\la}_{\pm}\|_{L^{\infty}_{t}H^{s-1}_{x}}=\|E^{\la}_{\pm}\|_{L^{\infty}_{t}H^{s-1}_{x}}\leq C, \qquad 
\|\partial_{t}\FF T_{\pm}E^{\la}_{\pm}\|_{L^{\infty}_{t}H^{s-1}_{x}}\leq C
\label{7.3.3}
\end{equation}
So we can conclude that that there exists two curl free vectors $E^{+}$, $E^{-}$ such that
\begin{equation}
T_{\pm}E^{\la}_{\pm}\rightarrow  E^{\pm} \quad \text{strongly in $C^{0}_{t}H^{s'-1}_{x,loc}$, for all $s'<s$}.
\label{7.3.4}
\end{equation}
Since we know that $T_{+}^{\la}$ is an isometry we have 
\begin{equation}
E^{\la}_{+}- T_{-}^{\la}E^{+}\rightarrow 0 \quad \text{and} \quad E^{\la}_{-}- T_{+}^{\la}E^{-}\rightarrow 0 \ \text{strongly in $C^{0}_{t}H^{s'-1}_{x,loc}$}.
\label{7.3.5}
\end{equation}
By Proposition \ref{7.1} we know that $\el=\el_{1}+\el_{2}$, so by using \eqref{7.3.5} we get
\begin{equation}
\la \el_{1}- T_{-}^{\la}E^{+}- T_{+}^{\la}E^{-}\longrightarrow 0 \quad \text{strongly in $C^{0}_{t}H^{s'-1}_{x,loc}$}
\label{7.3.6}
\end{equation}
In order to prove (c2) we use for $W^{\la}$, defined by \eqref{7.1.2}, a decomposition similar to \eqref{7.3.2}, namely
\begin{equation}
\hspace{-0.3cm}\FF W^{\la}_{+}(t,\xi)=\la e^{+ it/\la}\!\left(\frac{\FF \el(0,\xi)}{2i}-\frac{\la \FF \el_{t}(0,\xi)}{2}-\!\!
\int_{0}^{t} \!\frac{\FF F(s,\xi)}{2\la}e^{-it/\la}ds\right)
\label{7.3.7}
\end{equation}
and
\begin{equation}
\FF W^{\la}_{0}=\int_{0}^{t} \FF F(s,\xi)ds - \int_{0}^{t} \mathcal{H}^{\la} \el(s,\xi)ds,
\label{7.3.8}
\end{equation}
we define $W^{\la}_{-}$ in a similar way.
From \eqref{7.1.2}, \eqref{7.3.7}, \eqref{7.3.8} we have $W^{\la}=W^{\la}_{0}+W^{\la}_{+}+W^{\la}_{-}$. Arguing as before we have that there exists $W^{+}$ and $W^{-}$ such that
\begin{equation}
W^{\la}_{+}- T_{-}^{\la}W^{+}\rightarrow 0 \quad \text{and} \quad W^{\la}_{-}- T_{+}^{\la}W^{-}\rightarrow 0 \ \text{strongly in $C^{0}_{t}H^{s'-1}_{x,loc}$}
\label{7.3.9}
\end{equation}
and
$$W^{\la}_{0}\rightarrow 0 \ \text{strongly in $L^{\infty}_{t}H^{s-1}_{x}$}.$$
So we get that
\begin{equation}
 W^{\la}- T_{-}^{\la}W^{+}- T_{+}^{\la}W^{-}\rightarrow 0 \quad \text{strongly in $C^{0}_{t}H^{s'-1}_{x,loc}$}
\label{7.3.10}
\end{equation}
The last step is to identify who are $W^{\pm}$. Taking into account that $\FF T^{\la}_{+}\el_{+}=i \FF T^{\la}_{+}W^{\la}_{+}$ and \eqref{7.3.4} and \eqref{7.3.9} we end up with
\begin{equation}
W^{\pm}=-i E^{\pm}.
\label{7.3.11}
\end{equation}
\end{proof}

\subsection{Step 4: Equation of the correctors}
Finally in this section we finish the prove of the Main Theorem \ref{tm2} and we will able to show the equation \eqref{ceq} satisfied by the correctors. 
In order to do this we take the equation \eqref{E}
\begin{align}
\la^{2}\partial_{tt}\el+\el&=\dive\Delta^{-1}\nabla\dive\left(\rl\ul\otimes\ul +(\rl)^{\gamma}\mathbb{I}-\la^{2}\el\otimes\el\right)\notag\\
&+\frac{\la^{2}}{2}\dive\left(|\el |^{2}\mathbb{I}\right)-2\nabla\dive \ul,
\label{7.4.1}
\end{align}
we substitute the decompositions obtained in the previous sections and we send $\la$ to $0$.
Let $\phi\in\mathcal{D}((0,T)\times \R^{3})$, then we have
\begin{align}
(T^{\la}_{+}(\la^{2}\partial_{tt}\el+\el), \phi)&=(\el, \la^{2}\partial_{tt}T^{\la}_{-}\phi+T^{\la}_{-}\phi)\notag\\
&=(T^{\la}_{-}T^{\la}_{+}\el, \la^{2}\partial_{tt}T^{\la}_{-}\phi+T^{\la}_{-}\phi)\notag\\
&=(\la T^{\la}_{+}\el, \frac{1}{\la}(\la^{2}T^{\la}_{+}\partial_{tt}T^{\la}_{-}\phi+\phi)).
\end{align}
We know that $\la T^{\la}_{+}\el \rightharpoonup E^{+}$weakly in $L^{2}$ and we can compute
\begin{align}
\FF\left(\frac{\la^{2}T^{\la}_{+}\partial_{tt}T^{\la}_{-}\phi+T^{\la}_{-}\phi+\phi}{\la}\right)&=\la^{2}e^{-it/\la}\partial_{tt}(e^{it/\la}\FF \phi)+\frac{1}{\la}\FF\phi\notag\\
&=2i \partial_{t}\FF\phi +\la \partial_{tt}\FF\phi. 
\end{align}
So we have
\begin{align}
(T^{\la}_{+}(\la^{2}\partial_{tt}\el+\el), \phi)&=(\la \FF T^{\la}_{+}\el, \FF\frac{1}{\la}(\la^{2}T^{\la}_{+}\partial_{tt}T^{\la}_{-}\phi+\phi))\notag\\
&=(\la \FF T^{\la}_{+}\el, 2i \partial_{t}\FF\phi +\la \partial_{tt}\FF\phi)\notag\\
&\rightarrow (\FF E^{+}, 2i\partial_{t}\FF\phi)=-2i (\partial_{t}E^{+}, \phi).
\label{7.4.2}
\end{align}
The next term to analyze is the convective term 
$\rl\ul\otimes\ul=\rl(v^{\la}+W^{\la})\otimes (v^{\la}+W^{\la})$, it will be sufficient to analise the terms of this sort $\rl v^{\la}_{i}v^{\la}_{k}$, $\rl v^{\la}_{i}W^{\la}_{k}$, $\rl W^{\la}_{i}W^{\la}_{k}$, $i,k=1,\ldots,3$.
Since $v^{\la}_{i},\ v^{\la}_{k}$ are strongly convergent in $L^{2}$, by using the property (T2) we get
\begin{equation}
\label{7.4.3}
T^{\la}_{+}(\rl v^{\la}_{i}v^{\la}_{k})\rightharpoonup 0\quad \text{weakly in $L^{2}$}.
\end{equation}
Taking now into account (c2) we have
\begin{align}
\label{7.4.4}
&\lim_{\la\to 0}T^{\la}_{+}(\dive\Delta^{-1}\nabla\dive(\rl v^{\la}_{i}\otimes W^{\la}_{k})=\notag\\
&
-i\lim_{\la\to 0}T^{\la}_{+}((\dive\Delta^{-1}\nabla\dive(\rl v^{\la}_{i} T^{\la}_{+}E^{-}_{k}+\rl v^{\la}_{i} T^{\la}_{-}E^{+}_{k}))=I_{1}+I_{2}
\end{align}
For $I_{1}$ we have
\begin{align}
I_{1}&=-i\lim_{\la\to 0}T^{\la}_{+}(\dive\Delta^{-1}\nabla\dive(\rl v^{\la}_{i} T^{\la}_{+}E^{-}_{k}))\notag\\
& -i\lim_ {\la\to 0}(e^{-it/\la}\FF(\dive\Delta^{-1}\nabla\dive(\rl v^{\la}_{i}e^{-it/\la}\FF E^{-}_{k}))\to 0
\label{7.4.5a}
\end{align}
Concerning $I_{2}$ we have 
\begin{align}
I_{2}&=-i(\FF T^{\la}_{+}((\dive\Delta^{-1}\nabla\dive(\rl v^{\la}_{i} T^{\la}_{-}E^{+}_{k}), \FF \phi)\notag\\
&=-ie^{-it/\la}(\xi\FF(\rl v^{\la}_{i} )\ast\FF T^{\la}_{-}E^{+}_{k}, \FF \phi)=-i(\xi\FF(\rl v^{\la}_{i} )\ast\FF E^{+}_{k}, \FF \phi)\notag\\
&=-i(\dive(\rl v^{\la}_{i} E^{+}_{k}), \phi)\to -i(\dive(v_{i} E^{+}_{k}), \phi).
\label{7.4.5}
\end{align}
Finally we  estimate the term $\rl W^{\la}_{i}W^{\la}_{k}$, again we use (c2) and we obtain
\begin{align}
\lim_{\la\to 0}(T^{\la}_{+}(\rl W^{\la}_{i}W^{\la}_{k}), \phi)&=-(\FF T^{\la}_{+}[\rl(T^{\la}_{+}E^{-}+T^{\la}_{-}E^{+})_{i}(T^{\la}_{+}E^{-}+T^{\la}_{-}E^{+})_{k}], \FF \phi)\notag\\
&=\lim_{\la\to 0}(\FF(E^{+}+e^{-2it/\la}E^{-})_{i}\ast \FF(e^{it/\la}E^{+}+e^{-it/\la}E^{-})_{k}, \FF\phi)\notag\\
&=\lim_{\la\to 0}((E^{+}+e^{-2it/\la}E^{-})_{i}(e^{it/\la}E^{+}+e^{-it/\la}E^{-})_{k}, \phi)=0.
\label{7.4.6}
\end{align}
The next term that we have to estimate is $\nabla\dive \ul$, we have
\begin{align}
\lim_{\la\to 0}(T^{\la}_{+}\nabla\dive \ul, \phi)&=\lim_{\la\to 0}\left[(T^{\la}_{+}\nabla\dive v^{\la}, \phi)+(T^{\la}_{+}\nabla\dive W^{\la}, \phi)\right]\notag\\
&=-i\lim_{\la\to 0}(\FF T^{\la}_{+}(T^{\la}_{-}E^{+}+T^{\la}_{+}E^{-}), \FF \phi)\notag\\
&=-i\lim_{\la\to 0}(\xi_{i}\xi_{j}(\FF E^{+}+e^{-2it/\la}\FF E^{-}), \FF \phi)\notag\\
&=-i(\nabla\dive E^{+}, \phi).
\label{7.4.7}
\end{align}
By using \eqref{7.2.4a} we get that $\dive\Delta^{-1}\nabla\dive((\rl)^{\gamma}\mathbb{I})$ converges strongly to zero. It remains only to estimate the two terms $\dive\Delta^{-1}\nabla\dive(\la^{2}\el\otimes\el)$ and $\dive(|\la \el|^{2}\mathbb{I})$. Since the arguments works on a similar way we estimate only the first one. Let us denote by $\mathcal{A}$ the operator $\dive\Delta^{-1}\nabla\dive$ and by $a(\xi)$ its principal symbol. Then we have
\begin{align}
&(T^{\la}_{+}\mathcal{A}(\la^{2}\el\otimes\el), \phi)=\notag\\
(e^{-it/\la}a(\xi)\FF(T^{\la}_{-}E^{+}&+T^{\la}_{+}E^{-})\ast \FF(T^{\la}_{-}E^{+}+T^{\la}_{+}E^{-}), \FF\phi)=0.
\label{7.4.8}
\end{align}
If we sum up \eqref{7.4.2},  \eqref{7.4.3},  \eqref{7.4.2},  \eqref{7.4.5},  \eqref{7.4.5a},  \eqref{7.4.6},  \eqref{7.4.7},  \eqref{7.4.8} we end up with the following equation for the corrector $E^{+}$
\begin{equation}
\label{7.4.9}
\partial_{t} E^{+}+\dive(v \otimes E^{+})-\nabla\dive E^{+}=0, \qquad PE^{+}=0.
\end{equation} 
Moreover, by projecting the equation \eqref{7.4.1} in the divergence free space and following the same steps as berfore we get as $\la$ goes to $0$ the following relation for $E^{+}$
\begin{equation}
P\dive(v \otimes E^{+})=0.
\label{7.4.10}
\end{equation}
As a consequence of \eqref{7.4.9} and \eqref{7.4.10} we get that $E^{+}$ satisfies the following parabolic equation
\begin{equation}
\label{7.4.11}
\partial_{t} E^{+}-\Delta E^{+}+ Q\dive(v \otimes E^{+})=0.
\end{equation}
On order to obtain the equation satisfied by $E^{-}$ we can follow step by step what we have done for $E^{+}$.
From the previous paragraph it is clear that the equation \eqref{7.4.11} holds in the sense of distribution, in the next proposition we can establish a more precise result on the existence of solution for \eqref{7.4.11}, in particular we will see that the kernel of the Leray projector $P$ is an invariant subspace for the flow of the equation \eqref{7.4.11}.
\begin{proposition}
\label{p8}
Let us consider the correctors equation
\begin{equation}
\partial_{t}E^{\pm}-\Delta E^{\pm}+Q\dive(v \otimes E^{\pm})=0,
\label{7.4.12}
\end{equation}
where $v\in L^{\infty}(0,T;H^{s}(\R^{3}))$, $3/2\leq s\leq 2$ and the initial data satisfy
\begin{equation}
E^{\pm}(0)\in L^{2}(\R^{3}), \qquad PE^{\pm}(0)=0.
\label{7.4.13}
\end{equation}
Then the Cauchy problem \eqref{7.4.12}-\eqref{7.4.13} has a unique solution $E^{\pm}\in L^{\infty}(0,T;L^{2}(\R^{3}))$, such that  $PE^{\pm}(\cdot, t)=0$, for any $t\in [0,T]$.
\end{proposition}
\begin{proof}
The proof follows by rewriting \eqref{7.4.12} in integral form and by using a standard fixed point argument (for more detail see \cite{Tay81}, Chapter IV, Exercises 7.8 and 7.9).
\end{proof}

 
\bibliographystyle{amsplain}

\begin{thebibliography}{10}

\bibitem{Ada75}
R.~A. Adams, \emph{Sobolev spaces}, Academic Press, New York, 1975.


\bibitem{BG94}
 Y.~Brenier, E.~Grenier, \emph{Limite singuli\`ere du syst\`eme de {V}lasov-{P}oisson dans le r\'egime de quasi neutralit\'e: le cas ind\'ependant du temps}, C. R. Acad. Sci. Paris S\'er. I Math.,
 \textbf{318} (1994), no.~2, 121--124. 


\bibitem{CDMS96}
S.~Cordier, P.~Degond, P.~Markowich, and C.~Schmeiser, \emph{Travelling wave
  analysis of an isothermal {E}uler-{P}oisson model}, Ann. Fac. Sci. Toulouse
  Math. (6) \textbf{5} (1996), no.~4, 599--643.

\bibitem{CG00}
S.~Cordier and E.~Grenier, \emph{Quasineutral limit of an {E}uler-{P}oisson
  system arising from plasma physics}, Comm. Partial Differential Equations
  \textbf{25} (2000), no.~5-6, 1099--1113.

\bibitem{Don03}
D.~Donatelli, \emph{Local and global existence for the coupled
  {N}avier-{S}tokes-{P}oisson problem}, Quart. Appl. Math. \textbf{61} (2003),
  no.~2, 345--361.

\bibitem{DM08}
D.~Donatelli and P.~Marcati, \emph{A quasineutral type limit for the
  {N}avier-{S}tokes-{P}oisson system with large data}, Nonlinearity \textbf{21}
  (2008), no.~1, 135--148.

\bibitem{DFPS01}
B.~Ducomet, E.~Feireisl, H.~Petzeltov{\'a}, and I.~Stra{\v{s}}kraba,
  \emph{Existence globale pour un fluide barotrope autogravitant}, C. R. Acad.
  Sci. Paris S\'er. I Math. \textbf{332} (2001), no.~7, 627--632.

\bibitem{DFPS04}
B.~Ducomet, E.~Feireisl, H.~Petzeltov{\'a}, and I.~Stra{\v{s}}kraba, \emph{Global in time weak solutions for compressible barotropic
  self-gravitating fluids}, Discrete Contin. Dyn. Syst. \textbf{11} (2004),
  no.~1, 113--130.

\bibitem{GM01a}
I.~Gasser and P.~Marcati, \emph{The combined relaxation and vanishing {D}ebye
  length limit in the hydrodynamic model for semiconductors}, Math. Methods
  Appl. Sci. \textbf{24} (2001), no.~2, 81--92.

\bibitem{GM01b}
I.~Gasser and P.~Marcati, \emph{A vanishing {D}ebye length limit in a hydrodynamic model for
  semiconductors}, Hyperbolic problems: theory, numerics, applications, Vol. I,
  II (Magdeburg, 2000), Internat. Ser. Numer. Math., 140, vol. 141,
  Birkh\"auser, Basel, 2001, pp.~409--414.

\bibitem{GM03}
I.~Gasser and P.~Marcati, \emph{A quasi-neutral limit in the hydrodynamic model for charged
  fluids}, Monatsh. Math. \textbf{138} (2003), no.~3, 189--208. 

\bibitem{Ger91}
P.~G{\'e}rard, \emph{Microlocal defect measures}, Comm. Partial Differential
  Equations \textbf{16} (1991), no.~11, 1761--1794. 

\bibitem{GV95}
J.~Ginibre and G.~Velo, \emph{Generalized {S}trichartz inequalities for the
  wave equation}, J. Funct. Anal. \textbf{133} (1995), no.~1, 50--68.

\bibitem{GR95}
R.~J. Goldston and P.~H. Rutherford, \emph{Introduction to plasma physics},
  Institute of Physics Publishing, Bristol and Philadelphia, 1995.

\bibitem{G95a}
E.~Grenier, \emph{Defect measures of the {V}lasov-{P}oisson system in the
  quasineutral regime}, Comm. Partial Differential Equations \textbf{20}
  (1995), no.~7-8, 1189--1215.

\bibitem{G95}
E.~Grenier, \emph{Oscillations in quasineutral plasmas}, Comm. Partial
  Differential Equations \textbf{21} (1996), no.~3-4, 363--394.

\bibitem{HL09}
C.~Hao and H-L. Li, \emph{Global existence for compressible
  {N}avier-{S}tokes-{P}oisson equations in three and higher dimensions}, J.
  Differential Equations \textbf{246} (2009), no.~12, 4791--4812. 

\bibitem{JW06}
S.~Jiang and S.~Wang, \emph{The convergence of the {N}avier-{S}tokes-{P}oisson
  system to the incompressible {E}uler equations}, Comm. Partial Differential
  Equations \textbf{31} (2006), no.~4-6, 571--591.

\bibitem{JLW08}
Q.~Ju, F.~Li, and S.~Wang, \emph{Convergence of the {N}avier-{S}tokes-{P}oisson
  system to the incompressible {N}avier-{S}tokes equations}, J. Math. Phys.
  \textbf{49} (2008), no.~7, 073515, 8. 

\bibitem{KT98}
M.~Keel and T.~Tao, \emph{Endpoint {S}trichartz estimates}, Amer. J. Math.
  \textbf{120} (1998), no.~5, 955--980.

\bibitem{MZ10}
H-L. Li, A.~Matsumura, and G.~Zhang, \emph{Optimal decay rate of the
  compressible {N}avier-{S}tokes-{P}oisson system in {$\Bbb R^3$}}, Arch.
  Ration. Mech. Anal. \textbf{196} (2010), no.~2, 681--713. 

\bibitem{LPL96}
P.-L. Lions, \emph{Mathematical topics in fluid dynamics, incompressible
  models}, Claredon Press, Oxford Science Pubblications, 1996.

\bibitem{L-P.L.M98}
P.-L. Lions and N.~Masmoudi, \emph{Incompressible limit for a viscous
  compressible fluid}, J. Math. Pures Appl. (9) \textbf{77} (1998), no.~6,
  585--627.

\bibitem{L05}
G.~Loeper, \emph{Quasi-neutral limit of the {E}uler-{P}oisson and
  {E}uler-{M}onge-{A}mp\`ere systems}, Comm. Partial Differential Equations
  \textbf{30} (2005), no.~7-9, 1141--1167.

\bibitem{PWY06}
Y.-J. Peng, Y.-G.Wang, and W.-A. Yong, \emph{Quasi-neutral limit of the
  non-isentropic {E}uler-{P}oisson system}, Proc. Roy. Soc. Edinburgh Sect. A
  \textbf{136} (2006), no.~5, 1013--1026.

\bibitem{Si}
J.~Simon, \emph{Compact sets in the space {$L\sp p(0,T;B)$}}, Ann. Mat. Pura
  Appl. (4) \textbf{146} (1987), 65--96.

\bibitem{Ste93}
E.~M. Stein, \emph{Harmonic analysis: real-variable methods, orthogonality, and
  oscillatory integrals}, Princeton Mathematical Series, vol.~43, Princeton
  University Press, Princeton, NJ, 1993, With the assistance of Timothy S.
  Murphy, Monographs in Harmonic Analysis, III.

\bibitem{S77}
R.~S. Strichartz, \emph{Restrictions of {F}ourier transforms to quadratic
  surfaces and decay of solutions of wave equations}, Duke Math. J. \textbf{44}
  (1977), no.~3, 705--714.

\bibitem{Tar90}
L.~Tartar, \emph{{$H$}-measures, a new approach for studying homogenisation,
  oscillations and concentration effects in partial differential equations},
  Proc. Roy. Soc. Edinburgh Sect. A \textbf{115} (1990), no.~3-4, 193--230.

\bibitem{Tay81}
M.E. Taylor, \emph{Pseudodifferential operators}, Princeton Mathematical
  Series, vol.~34, Princeton University Press, 1981.

\bibitem{Tay91}
M.E. Taylor, \emph{Pseudodifferential operators and nonlinear {PDE}}, Birkh\"auser,
  1991.

\bibitem{W04}
S.~Wang, \emph{Quasineutral limit of {E}uler-{P}oisson system with and without
  viscosity}, Comm. Partial Differential Equations \textbf{29} (2004), no.~3-4,
  419--456.
\end{thebibliography}

\end{document}